\documentclass{amsart}
\usepackage[margin=1.5cm]{geometry}
\usepackage{amssymb,amsthm,comment,color,caption,enumerate,adjustbox,enumitem,tikz,
amsmath, amsfonts,amsthm,amssymb,amscd, verbatim, graphicx,color,multirow,booktabs, caption,tikz,tikz-cd, mathdots, todonotes}
\usepackage{colortbl, xcolor}
\usepackage{chngcntr, arydshln, soul, float,subcaption}
\numberwithin{equation}{section}
\usetikzlibrary{arrows,shapes,positioning,decorations.markings}
\usetikzlibrary{matrix, calc, arrows, graphs}
\usepackage{hyperref}
\newtheorem{theorem}{Theorem}[section]
\newtheorem{lemma}[theorem]{Lemma}
\newtheorem{proposition}[theorem]{Proposition}

\counterwithin{table}{section}
\begin{document}
\title[Number of subgroups]{On the maximum number of subgroups of a finite group}

\author[M. Fusari]{Marco Fusari}
\address{Marco Fusari, Dipartimento di Matematica ``Felice Casorati", University of Pavia, Via Ferrata 5, 27100 Pavia, Italy} 
\email{lucamarcofusari@gmail.com}

\author[P. Spiga]{Pablo Spiga}
\address{Pablo Spiga, Dipartimento di Matematica Pura e Applicata,\newline
 University of Milano-Bicocca, Via Cozzi 55, 20126 Milano, Italy} 
\email{pablo.spiga@unimib.it}
\maketitle
\begin{abstract}
Given a finite group $R$, we let $\mathrm{Sub}(R)$ denote the collection of all subgroups of $R$. We show that $|\mathrm{Sub}(R)|< c\cdot |R|^{\frac{\log_2(|R|)}{4}}$, where $c<7.372$ is an explicit absolute constant.  This result is asymptotically best possible. Indeed, as $|R|$ tends to infinity and $R$ is an elementary abelian $2$-group, the ratio $$\frac{|\mathrm{Sub}(R)|}{|R|^{\frac{\log_2(|R|)}{4}}}$$ tends to $c$.        	
\end{abstract}
          
\section{Introduction}
Let $R$ be a finite group of order $r$ and let $$\mathrm{Sub}(R)$$ be the family of subgroups of $R$.  It has been observed several times~\cite{BEJ,LPS,Sh,Sp} that, when $R$ is an elementary abelian $p$-group, there exist two constants $c_p$ and $c_p'$ depending on $p$ only such that
$$c_p\cdot r^{\frac{\log_p(r)}{4}}\le |\mathrm{Sub}(R)|\le c_p'\cdot r^{\frac{\log_p (r)}{4}}.$$ 
Borovik, Pyber and Shalev~\cite[Corollary 1.6]{BPS} have shown that, for an arbitrary finite group $R$ of order $r$, we have
$$|\mathrm{Sub}(R)|\le r^{\log_2 (r)\left( \frac{1}{4} +o(1)\right)}.$$
Therefore, in the light of the comment on elementary abelian $p$-groups, this bound is asymptotically best possible.

Besides a natural theoretic interest, there are several practical applications that require an explicit upper bound on $|\mathrm{Sub}(R)|$. In particular, these applications require to replace the ``$o(1)$'' appearing in the exponent of $r$ with an explicit constant. All applications that we are aware of come from the problem of classifying graphical regular representations of finite groups, see the introductory section of~\cite{Sp} for more details.

Now, let $p$ be a prime number and let
\begin{equation}\label{valueofcp}
c(p):=\prod_{i\ge 1}\frac{1}{1-\frac{1}{p^i}}\left(-1+2\sum_{k=0}^\infty\frac{1}{p^{k^2}}\right).
\end{equation}
In~\cite{Sp}, it was shown that, when $R$ is an elementary abelian $p$-group, $|\mathrm{Sub}(R)|$ is asymptotic to $c(p)|R|^{\log_p(|R|)/4}$ as $|R|$ tends to infinity. Moreover, using elementary methods, it was shown that, for a fixed prime $p$, when $R$ is an arbitrary group, $|\mathrm{Sub}(R)|\le c(2)|R|^{\log_2(|R|)/4+1.5315}$. In this paper, using more sophisticated methods, we improve this result.
\begin{theorem}\label{thrm:main}
Let $R$ be a finite group and let $\mathrm{Sub}(R)$ be the family of all subgroups of $R$. Then $|\mathrm{Sub}(R)|< c(2)|R|^{\frac{\log_2(|R|)}{4}}$.
\end{theorem}
We observe that $c(2)$ is approximately 
$$7.37197.$$

 \section{Preliminaries}\label{sec:ao}
Let $R$ be a non-identity finite group and let $r$ be the order of $R$. Let
\begin{equation*}
	r=\prod\limits_{i=1}^\ell p_i^{a_i}
\end{equation*}
be the prime factorization of $r$. In particular, $p_1,\ldots,p_\ell$ are distinct prime numbers and $a_i\geq 1$ for each $i\in\{1,\ldots,\ell\}$. Relabeling the indexed set if necessary, we may suppose that $$p_1 <p_2 <\cdots<p_\ell.$$ 

Let $H$ be a subgroup of $R$. Then $H$ is uniquely determined by a family $(Q_i)_i$ of Sylow $p_i$-subgroups of $H$, for each $i\in\{1,\ldots,\ell\}$. Each of these subgroups $Q_i$ is contained in a Sylow $p_i$-subgroup $P_i$ of $R$. From Sylow's theorems, all Sylow $p_i$-subgroups of $R$ are conjugate and hence $R$ has at most $r/{p_i}^{a_i}$ Sylow $p_i$-subgroups. Let $n_i$ be the number of Sylow $p_i$-subgroups of $R$. Then 
\begin{equation}\label{SylowConjugate}
	n_i=|R:{\bf N}_R(P_i)|=\frac{r}{|{\bf N}_R(P_i):P_i|p_i^{a_i}}\leq \frac{r}{p_i^{a_i}}.
\end{equation}
Therefore, we have  
\begin{equation}
	\prod\limits_{i=1}^\ell\frac{r}{|{\bf N}_R(P_i):P_i|p_i^{a_i}}=r^{\ell-1}\cdot\prod_{i=1}^\ell\frac{1}{|{\bf N}_R(P_i):P_i|}
\end{equation}
choices for the $\ell$-tuple $(P_i)_i$.
When $(P_i)_i$ is given, since $Q_i$ is a subgroup of $P_i$ and since $P_i$ is a $p_i$-group, we have at most 
\begin{equation*}
	\prod\limits_{i=1}^\ell S(p_i,a_i)
\end{equation*}
choices for the $\ell$-tuple $(Q_i)_i$, where $S(p_i,a_i)$ is the function defined in~\cite{Sp}. (The function $S(p,a)$ has the property that every $p$-group of order $p^a$ has at most $S(p,a)$ subgroups, see Lemma~\ref{lemma:new}. For not breaking the flow of the argument, we postpone the definition of $S(p,a)$ to~\eqref{valueofcp} in Section~\ref{sec:Spa}.) 
Thus $R$ has at most 
\begin{equation*}
	r^{\ell-1}\cdot \prod_{i=1}^\ell\frac{1}{|{\bf N}_R(P_i):P_i|}\cdot \prod\limits_{i=1}^\ell S(p_i,a_i)
\end{equation*}
subgroups. 

We are interested in studying when
\begin{equation}\label{Main0}
	r^{\ell-1}\cdot \prod_{i=1}^\ell\frac{1}{|{\bf N}_R(P_i):P_i|}\cdot \prod\limits_{i=1}^\ell S(p_i,a_i)< c(2)\cdot r^{\frac{\log_2{(r)}}{4}}
\end{equation}
holds, because in this case Theorem~\ref{thrm:main} immediately follows.

By taking $\log_r$ on both sides of~\eqref{Main0}, we obtain the equivalent inequality
\begin{equation*}
\ell-1-\sum_{i=1}^\ell\log_r(|{\bf N}_R(P_i):P_i|)+\sum_{i=1}^\ell\log_r(S(p_i,a_i))< \log_r(c(2))+\sum_{i=1}^\ell a_i\frac{\log_2(p_i)}{4}.
\end{equation*}
Finally, this can be rewritten in the following form
\begin{equation}\label{final}
	{\ell-1}-\sum_{i=1}^\ell\frac{\log(|{\bf N}_R(P_i):P_i|)}{\log (r)}<\frac{\log (c(2))}{\log (r)}+\sum_{i=1}^\ell\left(\frac{a_i\log (p_i)}{4\log (2)}-\frac{\log (S(p_i,a_i))}{\log (r)}\right).
\end{equation}
\subsection{The function $S(p,a)$}\label{sec:Spa}
Let $p$ be a prime number and let $a$ be a positive integer. We define
\begin{equation}\label{variousS}
S(p,a):=
\begin{cases}
2&\textrm{when }a:=1,\\
p+3&\textrm{when }a:=2,\\
2p^2+2p+4&\textrm{when }a:=3,\\
p^4+3p^3+4p^2+3p+5&\textrm{when }a:=4,\\
2p^6+2p^5+6p^4+6p^3+6p^2+4p+6&\textrm{when }a:=5,\\
c(p)p^{\frac{a^2}{4}}&\textrm{when }a\ge 6.
\end{cases}
\end{equation}
We observe that in~\cite{Sp}, the value of $c(p)$ is slightly different from the value we have defined in~\eqref{valueofcp}. In general, the value we have set for $c(p)$ in~\eqref{valueofcp} is less than or equal to the value of $c(p)$ in~\cite{Sp}, and it coincides with the value in~\cite{Sp} when $p=2$. All the results in~\cite{Sp} remain valid with this slightly improved constant. In particular, we have the following lemma.

\begin{lemma}\label{lemma:new}
Let $R$ be an elementary abelian $p$-group of order $p^a$. Then $|\mathrm{Sub}(R)|\le S(p,a)< c(2)|R|^{\frac{\log_2(|R|)}{4}}$. 
\end{lemma}
\begin{proof}
The proof follows from Remark~3.1 and Lemma~3.2 in~\cite{Sp}.
\end{proof}

From Lemma~\ref{lemma:new}, for the proof of Theorem~\ref{thrm:main}, we may suppose that $\ell\ge 2$. Indeed, throughout the rest of this paper, we tacitly assume $\ell\ge 2$. Furthermore, throughout the rest of this paper, we also tacitly assume the notation in this section.

We conclude this section with some numerical information.

\begin{lemma}\label{lemma:new-1}
For every positive integer $a$ and for every prime number $p$, we have $S(p,a)\le c(p)p^{a^2/4}$.
\end{lemma}
\begin{proof}
When $a\ge 6$, by~\eqref{variousS}, we have $S(p,a)= c(p)p^{a^2/4}$. When $a\le 5$, the proof follows from some computations. 

From~\eqref{valueofcp}, we have
$$c(p)> \frac{1}{1-\frac{1}{p}}\cdot \left(-1+2\left(1+\frac{1}{p}\right)\right)=\frac{p+2}{p-1}.$$

Assume $a=1$, that is, $S(p,a)=2$. We have $$2\le \frac{p+2}{p-1}\cdot p^{\frac{1}{4}}<c(p)\cdot p^{\frac{1}{4}}=c(p)p^{\frac{a^2}{4}},$$ for every $p$ (when $p\le 15$, the first inequality can be verified directly and, when $p\ge 16$, it is clear, because $p^{1/4}\ge 2$).  Assume $a=2$, that is, $S(p,a)=p+3$. We have $$p+3\le \frac{p+2}{p-1}\cdot p<c(p)\cdot p=c(p)p^{\frac{a^2}{4}},$$ for every $p$ (the first inequality follows from an easy computation). Assume $a=3$, that is, $S(p,a)=2p^2+2p+4$. We have $$2p^2+2p+4\le \frac{p+2}{p-1}\cdot p^{\frac{9}{4}}<c(p)\cdot p^{\frac{9}{4}}=c(p)p^{\frac{a^2}{4}},$$ for every $p$ (when $p\le 15$, the first inequality can be verified directly  and, when $p\ge 16$, it is clear, because $p^{1/4}\ge 2$ and $p^2+p+2\le (p+2)p^2/(p-1)$). 

To deal with larger values of $a$, we need to refine the estimate on $c(p)$. 
From~\eqref{valueofcp}, we have
$$c(p)> \frac{1}{\left(1-\frac{1}{p}\right)\left(1-\frac{1}{p^2}\right)}\cdot \left(-1+2\left(1+\frac{1}{p}\right)\right)=\frac{p^3+2p^2}{(p-1)(p^2-1)}.$$
Assume $a=4$, that is, $S(p,a)=p^4+3p^3+4p^2+3p+5$. We have $$p^4+3p^3+4p^2+3p+5\le \frac{p^3+2p^2}{(p-1)(p^2-1)}\cdot p^{4}<c(p)\cdot p^4=c(p)p^{\frac{a^2}{4}},$$ for every $p$ (the first inequality follows form a computation by expanding the two members and manipulating the two polynomials in $p$). Finally, assume $a=5$, that is, $S(p,a)=2p^6+2p^5+6p^4+6p^3+6p^2+4p+6$. We have $$2p^6+2p^5+6p^4+6p^3+6p^2+4p+6\le \frac{p^3+2p^2}{(p-1)(p^2-1)}\cdot p^{\frac{25}{4}}<c(p)\cdot p^{\frac{25}{4}}=c(p)p^{\frac{a^2}{4}},$$ for every $p$ (when $p\le 15$, the first inequality can be verified directly  and, when $p\ge 16$, it follows with a computation, because $p^{1/4}\ge 2$ and $p^6+p^5+3p^4+3p^3+3p^2+2p+3\le (p^3+2p^2)p^6/((p-1)(p^2-1))$).

\end{proof}

\subsection{Small groups and computations}\label{magmacomputations}
To avoid some long arguments for groups having small order, we have checked the veracity of Theorem~\ref{thrm:main} with the computer algebra system \texttt{magma}~\cite{magma}, for all groups $R$ with $|R|\le 2\,000$. Indeed, the library of ``small groups'' in \texttt{magma} has an exhaustive list of all finite groups of order at most $2\,000$.

\section{Solvable groups}\label{sec:solvable}
In this section, we prove the following result.
\begin{theorem}{\label{SolvableTheorem}}
If $R$ is solvable, then $|\mathrm{Sub}(R)|< c(2)\cdot|R|^{\frac{\log_2{(|R|)}}{4}}$.
\end{theorem}
We argue by induction on $|R|=r$, where the base case of the induction can be considered Lemma~\ref{lemma:new}.
Let $P_\ell$ be a Sylow $p_\ell$-subgroup of $R$. Since $R$ is solvable, $R$ admits a Hall $p_\ell'$-subgroup $K$. Using $P_\ell$ and $K$ we estimate the number of subgroups of $R$.

Let $H$ be a subgroup of $R$. Then $H=H_{p_\ell'}H_{p_\ell}$, where $H_{p_\ell'}$ is a Hall $p_\ell'$-subgroup of $H$ and $H_{p_\ell}$ is a Sylow $p_\ell$-subgroup of $H$. Now, from Hall's theorem, $H_{p_\ell'}$ is conjugate, via an element of $P_\ell$, to a subgroup of $K$ and, from Sylow's theorem, $H_{p_\ell}$ is conjugate, via an element of $K$, to a subgroup of $P_\ell$. In particular, we have at most $$1+(|\mathrm{Sub}(K)|-1)p_\ell^{a_\ell}$$ choices for $H_{p_\ell'}$, because every non-identity subgroup of $K$ as at most $p_\ell^{a_\ell}$ conjugates. Similarly, we have at most $$1+(|\mathrm{Sub}(P_\ell)|-1)\frac{r}{p_\ell^{a_\ell}}$$ choices for $H_{p_\ell}$, because every non-identity subgroup of $P_{\ell}$ as at most $|K|=r/p_\ell^{a_\ell}$ conjugates. Therefore,
\begin{equation}\label{eq:solvable0}
|\mathrm{Sub}(R)|\le (1+(|\mathrm{Sub}(K)|-1) p_\ell^{a_\ell})\cdot \left(
1+(|\mathrm{Sub}(P_\ell)|-1)\frac{r}{p_\ell^{a_\ell}}\right).
\end{equation}
Thus
$$|\mathrm{Sub}(R)|< |\mathrm{Sub}(K)|p_\ell^{a_\ell}\cdot |\mathrm{Sub}(P_\ell)|\frac{r}{p_\ell^{a_\ell}}.$$
Using Lemma~\ref{lemma:new}, this can be simplified in
$$|\mathrm{Sub}(R)|\le |\mathrm{Sub}(K)|S(p_\ell,a_\ell)r.$$
As $|K|<r$, by induction we may bound the number of subgroups of $K$ by $c(2)|K|^{\log_2(|K|)/4}$ and hence we obtain
\begin{align}\label{eq:solvable}
|\mathrm{Sub}(R)|&<c(2)(r/p_\ell^{a_\ell})^{\frac{\log_2(r/p_\ell^{a_\ell})}{4}}S(p_\ell,a_\ell)r=
c(2)r^{\frac{\log_2(r)}{4}}\cdot\frac{S(p_\ell,a_\ell)r}{p_\ell^{\frac{a_\ell\log_2(r/p_\ell^{a_\ell})}{4}}r^{\frac{\log_2(p_\ell^{a_\ell})}{4}}}\\\nonumber
&=
c(2)r^{\frac{\log_2(r)}{4}}\cdot\frac{S(p_\ell,a_\ell)r}{(r^2/p_\ell^{a_\ell})^{\frac{\log_2(p_\ell^{a_\ell})}{4}}}.
\end{align}

\begin{proposition}\label{proposition:appendix}
Let $r=p_1^{a_1}\cdots p_\ell^{a_\ell}$, with $p_1<\cdots<p_\ell$, $a_1,\ldots,a_{\ell}\ge 1$ and $\ell \ge 2$, and let $i\in \{1,\ldots,\ell\}$. Then
\begin{equation}\label{eq:vera}
\frac{S(p_i,a_i)r}{(r^{2}/p_i^{a_i})^{a_i\frac{\log_2(p_i)}{4}}}\le 1.\end{equation}
except when  
\begin{enumerate}
\item\label{appendix0}$a_i=1$ and
\begin{itemize}
\item $p_i\in\{2,3\}$,
\item $p_i=5$ and $r/p_i \le 4\,918$,
\item $p_i=7$ and $r/p_i\le 23$,
\item $p_i=11$ and $r=2\cdot 11$, $r=3\cdot 11$,
\item $p_i=13$ and $r=2\cdot 13$, 
\end{itemize}
\item\label{appendix1}$a_i=2$ and
\begin{itemize}
\item $p_i=2$,
\item $p_i=3$ and $r/p_i^2\le 46$,
\item $p_i=5$ and $r=50=2\cdot 5^2$, or $r=75=3\cdot 5^2$,
\end{itemize}
\item\label{appendix2}$a_i=3$ and
\begin{itemize}
\item $p_i=2$ and $r/p_i^3\le 723$,
\item $p_i=3$ and $r/p_i^3\le 7$,
\end{itemize}
\item\label{appendix3}$a_i=4$ and
\begin{itemize}
\item $p_i=2$ and $r/p_i^4\le 67$,
\item $p_i=3$ and $r=2\cdot 3^4=162$,
\end{itemize}
\item\label{appendix4}$a_i=5$ and
\begin{itemize}
\item $p_i=2$ and $r/p_i^5\le 29$,
\item $p_i=3$ and $r=2\cdot 3^5=486$.
\end{itemize}
\item\label{appendix5}$a_i\ge 6$, $p_i=2$ and $r/p_i^{a_i}\le 21$.
\end{enumerate}
\end{proposition}
For not breaking the flow of the argument we postpone the proof of Proposition~\ref{proposition:appendix} to Section~\ref{sec:appendixproposition}.

To conclude the proof of Theorem~\ref{SolvableTheorem} we consider various cases, depending on whether $a_\ell=1$ or $a_\ell\ge 2$.

\begin{lemma}\label{sol1}
If $a_\ell=1$, then Theorem~$\ref{SolvableTheorem}$ holds true.
\end{lemma}

\begin{proof}
From Proposition~\ref{proposition:appendix}, Theorem~\ref{SolvableTheorem} follows from~\eqref{eq:solvable}, except when part~\eqref{appendix0} holds. When $r=|R|\le 2\, 000$, the veracity of this lemma follows from Section~\ref{magmacomputations}. Therefore, for the rest of the proof, we may suppose that $r>2\, 000$. In particular, either $p_\ell=3$ and $r=2^{a_1}\cdot 3$, or $p_\ell=5$ and $r=r'\cdot 5$  with $400< r'\le 4\,918$. 

We first need to refine~\eqref{eq:solvable} (for the cases under consideration). 
When $a_\ell=1$, we have $S(p_\ell,a_\ell)=2$ and hence~\eqref{eq:solvable0} becomes
\begin{equation*}
        |\mathrm{Sub}(R)|\leq (|\mathrm{Sub}(K)|p_\ell+1-p_\ell)\left(\frac{r}{p_\ell}+1\right).
\end{equation*}
As $|K|<r$, arguing by induction, we deduce
\begin{align*}
|\mathrm{Sub}(R)|&
\leq (c(2)(r/p_\ell)^{\frac{\log_2(r/p_\ell)}{4}}p_\ell+1-p_\ell)\left(\frac{r}{p_\ell}+1\right)\\
&< c(2)(r/p_\ell)^{\frac{\log_2(r/p_\ell)}{4}}p_\ell(r/p_\ell+1)\le c(2)r^{\frac{\log_2(r)}{4}}\cdot \frac{r+p_\ell}{p_\ell^{\frac{\log_2(r/p_\ell)}{4}}r^{\frac{\log_2 (p_\ell)}{4}}}\\
&\le c(2)r^{\frac{\log_2(r)}{4}}\cdot \frac{r+p_\ell}{2^{\frac{\log_2(p_\ell)\log_2(r/p_\ell)}{4}}r^{\frac{\log_2 (p_\ell)}{4}}}=
 c(2)r^{\frac{\log_2(r)}{4}}\cdot \frac{r+p_\ell}{(r^2/p_\ell)^{\frac{\log_2(p_\ell)}{4}}}.
\end{align*}
(These computations have allowed to replace the numerator $S(p_\ell,a_\ell)r=2r$ appearing in~\eqref{eq:solvable} with $r+p_\ell$.)
In particular, the lemma follows as long as 
\begin{equation}\label{Aleja}\frac{r+p_\ell}{(r^2/p_\ell)^{\frac{\log_2(p_\ell)}{4}}}\le 1.\end{equation}

Assume $p_\ell=5$. Since we are assuming $r\ge 2\,000$, we have $r/p_\ell\ge 256=2^8$ and hence we obtain
\begin{align*}
\frac{r+p_\ell}{(r^2/p_\ell)^{\frac{\log_2(p_\ell)}{4}}}&=\left(1+\frac{5}{r}\right)r^{1-\frac{\log_2(5)}{2}}5^{\frac{\log_2(5)}{4}}=\left(1+\frac{5}{r}\right)5^{1-\frac{\log_2(5)}{2}}\left(\frac{r}{5}\right)^{1-\frac{\log_2(5)}{2}}5^{\frac{\log_2(5)}{4}}\\
&\le \left(1+\frac{5}{r}\right)5^{1-\frac{\log_2(5)}{4}}256^{1-\frac{\log_2(5)}{2}}=\left(1+\frac{5}{r}\right)2^{\log_2(5)-\frac{(\log_2 (5))^2}{4}}2^{8-4\log_2(5)}\\
&=\left(1+\frac{5}{r}\right)2^{8-3\log_2(5)-\frac{(\log_2 (5))^2}{4}}\le \left(1+\frac{5}{r}\right)\cdot 0.81\le \left(1+\frac{1}{256}\right)\cdot 0.81<1.
\end{align*}
 Therefore, for the rest of the proof we suppose $p_\ell=3$. In particular, $\ell=2$, $p_1=2$ and $r=2^{a_1}\cdot 3$.

Let $P$ be a Sylow $2$-subgroup of $R$ and let  $T$ be a Sylow $3$-subgroup of $R$. Thus $|P|=2^{a_1}$ and $|T|=3$.
Let $H$ be an arbitrary subgroup of $R$. Then $H=\langle Q,S\rangle$, where $Q$ is a Sylow $2$-subgroup of $H$ and $S$ is a Sylow $3$-subgroup of $H$.
If $S=1$, then we have at most
\begin{equation}\label{eq:3}
c(2)\cdot 3\cdot (2^{a_1})^{\frac{\log_2(2^{a_1})}{4}}
=
c(2)\cdot 3\cdot 2^{\frac{a_1^2}{4}}
\end{equation}
choices for $H=Q$, because we have at most $3$ Sylow $2$-subgroups in $R$. Assume that $S\not=1$.
Let $\varepsilon\in \{1,3\}$ be the number of Sylow $2$-subgroups of $R$ and let $P_1,\ldots,P_\varepsilon$ be the Sylow $2$-subgroups of $R$ with $P=P_1$. Now, $Q\le P_i$ for some $i\in \{1,\ldots,\varepsilon\}$. As $S$ acts transitively by conjugation on the set $\{P_1,\ldots,P_\varepsilon\}$ of Sylow $2$-subgroups of $R$, replacing $Q$ by a suitable $S$-conjugate, we may suppose that $Q\le P_1=P$.

Let $a\in \{0,\ldots,a_1\}$. Corollary~4.2 in~\cite{Sh} shows that the number of subgroups of
$P$ having index $p^a$ is at most
\[
\left[
\begin{array}{c}
{a_1}\\
a
\end{array}
\right]_2=\frac{(2^{a_1}-1)\cdots (2^{{a_1}-a+1}-1)}{(2^1-1)\cdots (2^a-1)}.
\]
(Here, we are denoting with $\left[
\begin{array}{c}
{a_1}\\
a
\end{array}
\right]_2$ the $2$-binomial coefficient.)
Wince $H=\langle Q,S\rangle=\langle Q,S^x\rangle$, $\forall x\in Q$, we may replace $S$ with any $Q$-conjugate. We deduce that the number of subgroups of $R$ having order divisible by $3$ is at most
\begin{equation}\label{eq:5}
\sum_{a=0}^{a_1}
\left[
\begin{array}{c}
{a_1}\\
a
\end{array}
\right]_2\cdot 2^{a}.
\end{equation}
For $a\in \{1,\ldots,a_1\}$, we have
\[
\left[
\begin{array}{c}
a_1\\
a
\end{array}
\right]_2\cdot 2^{a}=
(2^{a_1}-1)\left[
\begin{array}{c}
{a_1}-1\\
a-1
\end{array}
\right]_2+\left[
\begin{array}{c}
{a_1}\\
a
\end{array}
\right]_2.
\]
Therefore,~\eqref{eq:5} becomes
\begin{align}\label{eq:_3}
1+(2^{a_1}-1)\sum_{a=1}^{a_1}\left[
\begin{array}{c}
{a_1}-1\\
a-1
\end{array}
\right]_2+
\sum_{a=1}^{a_1}
\left[
\begin{array}{c}
{a_1}\\
a
\end{array}
\right]_2
&=
(2^{a_1}-1)\sum_{a=0}^{a_1-1}\left[
\begin{array}{c}
a_1-1\\
a
\end{array}
\right]_2+
\sum_{a=0}^{a_1}
\left[
\begin{array}{c}
a_1\\
a
\end{array}
\right]_2.
\\\nonumber
\end{align}
Since $$\sum_{a=0}^{a_1-1}\left[
\begin{array}{c}
a_1-1\\
a
\end{array}
\right]_2 \hbox{ and }
\sum_{a=0}^{a_1}
\left[
\begin{array}{c}
a_1\\
a
\end{array}
\right]_2$$
count the number of subspaces of a vector space of dimension $a_1-1$ and $a_1$ over the field with $2$ elements, from Lemma~\ref{lemma:new}, we deduce that~\eqref{eq:5} is at most
$$
c(2)\cdot(2^{a_1}-1)2^{\frac{(a_1-1)^2}{4}}+c(2)\cdot 2^{\frac{a_1^2}{4}}.$$

Summing up, from~\eqref{eq:3} and~\eqref{eq:_3}, the number of subgroups of $R$ is at most
\begin{equation}\label{elisa}
c(2)\cdot 3\cdot 2^\frac{a_1^2}{4}+
c(2)\cdot(2^{a_1}-1)2^{\frac{(a_1-1)^2}{4}}+c(2)\cdot 2^{\frac{a_1^2}{4}}=
c(2)\cdot 2^{\frac{a_1^2}{4}}\cdot \left(4+2^{\frac{a_1}{2}+\frac{1}{4}}-2^{-\frac{a_1}{2}+\frac{1}{4}}\right).
\end{equation}

Assume first $a_1\ge 12$. Then, from~\eqref{elisa}, we obtain
\begin{equation}\label{elisa1}|\mathrm{Sub}(R)|< c(2)\cdot 2^{\frac{a_1^2}{4}}\cdot \left(4+2^{\frac{a_1}{2}+\frac{1}{4}}\right)< c(2)\cdot 2^{\frac{a_1^2}{4}}\cdot 2^{\frac{a_1}{2}+\frac{1}{3}}=c(2)\cdot 2^{\frac{a_1^2}{4}+\frac{a_1}{2}+\frac{1}{3}},
\end{equation}
where the last inequality follows with a computation using $a_1\ge 12$. On the other hand, we have
\begin{align}\label{elisa2}
c(2)r^{\frac{\log_2(r)}{4}}=c(2)\cdot (3\cdot 2^{a_1})^{\frac{\log_2(3\cdot 2^{a_1})}{4}}&\ge c(2)\cdot (2^{a_1+1})^{\frac{\log_2(3\cdot 2^{a_1})}{4}}=c(2)\cdot 2^{\frac{a_1^2}{4}+a_1\left(\frac{1}{4}+\frac{\log_2(3)}{4}\right)+\frac{\log_2(3)}{4}}.
\end{align}
Now, observe $$\frac{1}{4}+\frac{\log_2(3)}{4}\ge \frac{1}{4}+\frac{1}{4}=\frac{1}{2}$$ and $\log_2(3)/4\ge 1/3$. Therefore, when $a_1\ge 12$, the result follows from~\eqref{elisa1} and~\eqref{elisa2}.

When $3\le a_1\le 11$, we have verified with a calculator that~\eqref{elisa} is less than or equal to $c(2)\cdot (3\cdot 2^{a_1})^{\log_2(3\cdot 2^{a_1})/4}$ and hence the result follows also in this case. Finally, when $a_1\le 2$, we have $r\le 12\le 2,\,000$.
\end{proof}

\begin{proof}[Proof of Theorem~$\ref{SolvableTheorem}$]
From Lemma~\ref{sol1}, we may suppose that $a_\ell\ge 2$. From Proposition~\ref{proposition:appendix}, Theorem~\ref{SolvableTheorem} follows from~\eqref{eq:solvable}, except when one of part~\eqref{appendix1}--\eqref{appendix5} holds. Observe that we are applying Proposition~\ref{proposition:appendix} with $i=\ell$ and hence $p_i\ne 2$. Thus $|R|\le 486$.  We have verified the veracity of this lemma with a computation using the database of small groups in the computer algebra system \texttt{magma}~\cite{magma}, see Section~\ref{magmacomputations}.
\end{proof}

\section{Notation and arithmetic reductions}
\subsection{Notation}\label{sec:notation}
In the light of Theorem~\ref{SolvableTheorem}, for the rest of our argument we may suppose that $R$ is not solvable. In particular, from the Odd Order Theorem, we have 
\begin{equation}\label{A}p_1=2.\end{equation} Clearly, \begin{equation}\label{B}a_1\ge 2\end{equation} because a non-abelian simple group cannot have a cyclic Sylow $2$-subgroup. Moreover, \begin{equation}\label{C}\ell\ge 3\end{equation} from the celebrated $p^\alpha q^\beta$ theorem of Burnside. 

Recall from Section~\ref{sec:ao} that, for each prime $p_i$, $P_i$ is a Sylow $p_i$-subgroup of $R$. From~\cite{GMN}, we have that, if $p_i\ge 5$, then $|{\bf N}_R(P_i):P_i|\ne 1$. In particular, for each prime $p_i\ge 5$, we have $|{\bf N}_R(P_i):P_i|\ge 2$. From~\cite{GMN}, we see that the same conclusion holds when $p_i=3$, except (possibly) when $\mathrm{PSL}_2(3^{3^a})$ is a composition factor of $R$ for some $a\ge 1$. Hence we may replace~\eqref{final} with the inequality
\begin{equation}\label{finall}
	(\ell-1)\left(1-\frac{\log (2)}{\log (r)}\right)+\varepsilon\frac{\log (2)}{\log (r)}<\frac{\log (c(2))}{\log (r)}+\sum_{i=1}^\ell\left(\frac{a_i\log (p_i)}{4\log (2)}-\frac{\log (S(p_i,a_i))}{\log (r)}\right),
\end{equation}
where $\varepsilon=0$ when $R$ has no composition factor isomorphic to $\mathrm{PSL}_2(3^{3^a})$ and $\varepsilon=1$ otherwise.

From~\eqref{finall}, we are interested in the function
\begin{equation}\label{function:f}
f(r):=\frac{\log (c(2))}{\log (r)}+\sum_{i=1}^\ell\left(\frac{a_i\log (p_i)}{4\log (2)}-\frac{\log (S(p_i,a_i))}{\log (r)}\right)-
(\ell-1)+(\ell-2)\frac{\log (2)}{\log (r)}.
\end{equation}
 Because of the peculiar behavior of $S(p_i,a_i)$, when $a_i\le 5$, we consider the auxiliary function $$\mathcal{S}(p_i,a_i)=c(p_i)p_i^{\frac{a_i^2}{4}}$$ and 
\begin{equation}\label{function:ffttt}\mathtt{f}(r):=
\frac{\log (c(2))}{\log (r)}+\sum_{i=1}^\ell\left(\frac{a_i\log (p_i)}{4\log (2)}-\frac{\log (\mathcal{S}(p_i,a_i))}{\log (r)}\right)-
(\ell-1)+(\ell-2)\frac{\log (2)}{\log (r)}.\end{equation}

\subsection{Arithmetic reductions}\label{sec:artithmetic}

We say that $r$ is good if $f(r)> 0$ and we say that $r$ is $\mathtt{good}$ if $\mathtt{f}(r)> 0$. From Lemma~\ref{lemma:new-1}, $S(p,a)\le\mathcal{S}(p,a)$ and hence $f(r)\ge \mathtt{f}(r)$. In particular, if $r$ is $\mathtt{good}$, then $r$ is good. Observe that, when $r=|R|$ is good, Theorem~\ref{thrm:main} follows immediately from the discussion in Section~\ref{sec:notation}.

We use elementary calculus to deduce some important facts about $f(r)$.

\begin{lemma}\label{direction1}
Assume~$\eqref{A}$,~$\eqref{B}$ and~$\eqref{C}$. Let $i\in \{1,\ldots,\ell\}$ and let $r'$ be the positive integer obtained from $r$, by replacing the prime $p_i$ with a prime number $p>p_i$ and with $p\notin\{p_1,\ldots,p_\ell\}$. If $r$ is good, then so is $r'$.
\end{lemma}   

\begin{lemma}\label{direction2}
Assume~$\eqref{A}$,~$\eqref{B}$ and~$\eqref{C}$. Let $p$ be a prime number with $p\notin \{p_1,\ldots,p_\ell\}$ and $p\ge 17$ and let $r'=r\cdot p$. If $r$ is good, then so is $r'$.
\end{lemma}

\begin{lemma}\label{direction3}
Assume~$\eqref{A}$,~$\eqref{B}$ and~$\eqref{C}$. Let $i\in \{1,\ldots,\ell\}$  and let $r'=r\cdot p_i$. If $r$ is $\mathtt{good}$, then so is $r'$.
\end{lemma}

We prove Lemma~\ref{direction1} in Section~\ref{sec:directionalmonotonicity}, we prove Lemma~\ref{direction2} in Section~\ref{sec:direction2} and we prove Lemma~\ref{direction3} in Section~\ref{sec:direction3}.

Using Lemmas~\ref{direction1},~\ref{direction2} and~\ref{direction3}, we are able to reduce the proof of Theorem~\ref{thrm:main} to a very limited number of cases.
\begin{proposition}\label{lmax}
If $r$ satisfies any of the following conditions, then $r$ is good. In particular, if $|R|$ satisfies any of the following conditions, then	$|\mathrm{Sub}(R)|< c(2)\cdot |R|^{\log_2(|R|)/4}$. 
\begin{enumerate}
\item\label{eq:AA} $\ell\ge 13$;

\item\label{eq:BB}  $\ell=4$; moreover, $p_\ell\ge 79$, or $a_i\ge 5$ for some $i\in \{2,\ldots,\ell\}$, or $a_1\ge 28$;

\item  $\ell=5$; moreover, $p_\ell\ge 173$, or $a_i\ge 5$ for some $i\in \{2,\ldots,\ell\}$, or $a_1\ge 15$;

\item  $\ell=6$; moreover, $p_\ell\ge 251$, or $a_i\ge 5$ for some $i\in \{2,\ldots,\ell\}$, or $a_1\ge 12$;

\item  $\ell=7$; moreover, $p_\ell\ge 307$, or $a_i\ge 5$ for some $i\in \{2,\ldots,\ell\}$, or $a_1\ge 10$;
\item  $\ell=8$; moreover, $p_\ell\ge 277$, or $a_i\ge 5$ for some $i\in \{2,\ldots,\ell\}$, or $a_1\ge 9$;

\item  $\ell=9$; moreover, $p_\ell\ge 233$, or $a_i\ge 4$ for some $i\in \{2,\ldots,\ell\}$, or $a_1\ge 7$;

\item  $\ell=10$; moreover, $p_\ell\ge 163$, or $a_i\ge 3$ for some $i\in \{2,\ldots,\ell\}$, or $a_1\ge 6$;

\item\label{eq:BBBBCCCC}  $\ell=11$; moreover, $p_\ell\ge 89$, or $a_i\ge 3$ for some $i\in \{2,\ldots,\ell\}$, or $a_1\ge 4$;

\item\label{eq:BBBB}  $\ell=12$; moreover, $p_\ell\ge 47$, or $a_i\ge 2$ for some $i\in \{2,\ldots,\ell\}$, or $a_1\ge 3$.
\end{enumerate}
	\begin{proof}
We have implemented the functions $f(r)$ and $\mathtt{f}(r)$  in~\eqref{function:f} and in~\eqref{function:ffttt} in a computer. 

We have verified that  $$s=2^2\cdot 3 \cdot 5 \cdot 7 \cdot 11 \cdot 13 \cdot 17 \cdot 19 \cdot 23 \cdot 29 \cdot 31 \cdot 37\cdot 41$$ is $\mathtt{good}$. Observe that $s$ is divisible by the first $13$ prime numbers. From Lemmas~\ref{direction1},~\ref{direction2} and~\ref{direction3}, we deduce that any positive integer $r$ with $\ell>12$ is good. This proves~\eqref{eq:AA}.

Next, we prove~\eqref{eq:BBBB}. We have verified that  
\begin{align*}
&2^2\cdot 3 \cdot 5 \cdot 7\cdot 11\cdot 13\cdot 17\cdot 19\cdot 23\cdot 29\cdot 31\cdot 67,\\
& 2^2\cdot 3^2 \cdot 5 \cdot 7\cdot 11\cdot 13\cdot 17\cdot 19\cdot 23\cdot 29\cdot 31\cdot 37,\\ 
&2^3\cdot 3 \cdot 5 \cdot 7\cdot 11\cdot 13\cdot 17\cdot 19\cdot 23\cdot 29\cdot 31\cdot 37
\end{align*} are $\mathtt{good}$. From Lemmas~\ref{direction1} and~\ref{direction3}, we deduce that any positive integer $r$ with $\ell=12$ and $p_\ell\ge 67$, or with $a_i\ge 2$ for some $i\ge 2$, or with $a_1\ge 3$ is $\mathtt{good}$. In particular, in all of these cases $r$ is also good. Now, to obtain the refined condition stated in~\eqref{eq:BBBB}, we have computed explicitly the function $f$ in all numbers $r$ of the form
$r=2^2\cdot 3 \cdot 5 \cdot 7\cdot 11\cdot 13\cdot 17\cdot 19\cdot 23\cdot 29\cdot 31\cdot p_{12}$ with $p_{12}\le 67$.

All other parts are proved similarly.
	\end{proof}
\end{proposition}

\section{The case $\ell=3$}\label{sec:l=3}
In this section we prove Theorem~\ref{thrm:main} when $\ell=3$; this is a case where Proposition~\ref{lmax} gives no information.

\begin{lemma}\label{simple}Let $R$ be a non-abelian simple group whose order is divisible by at most three distinct primes. Then one of the following holds 
\begin{itemize}
\item $R\cong \mathrm{Alt}(5)\cong\mathrm{PSL}_2(4)\cong\mathrm{PSL}_2(5)$ and $|R|=2^2\cdot 3\cdot 5=60$,  
\item $R\cong \mathrm{PSL}_3(2)\cong\mathrm{PSL}_2(7)$ and $|R|=2^3\cdot 3\cdot 7=168$, 
\item $R\cong \mathrm{Alt}(6)\cong\mathrm{PSL}_2(9)$ and $|R|=2^3\cdot 3^2\cdot 5=360$, 
\item $R\cong \mathrm{PSL}_2(8)$ and $|R|=2^3\cdot 3^2\cdot 7=504$, 
\item $R\cong \mathrm{PSL}_2(17)$ and $|R|=2^4\cdot 3^2\cdot 17=2\,448$, 
\item $R\cong\mathrm{PSL}_3(3)$ and $|R|=2^4\cdot 3^3\cdot 13=5\,616$,
\item $R\cong\mathrm{PSU}_3(3)$ and $|R|=2^5\cdot 3^3\cdot 7=6\,048$,
\item $R\cong \mathrm{PSU}_4(2)\cong\mathrm{PSp}_4(3)$ and $|R|=2^6\cdot 3^4\cdot 5=25\,920$.
\end{itemize}
\end{lemma}
\begin{proof}
This result follows from the contribution of various authors (Brauer, Herzog, Klinger, Leon, Mason, Thompson, and Wales) and we refer to~\cite{BCM} and to the references therein for more details.
\end{proof}

We also need a few rather technical observations.
\begin{lemma}\label{bounds:technical}
Let $r\in\mathbb{N}$ with $r\ge 2\,000$, then
\begin{enumerate}
\item\label{bounds:technical1}$4r(r/7)^{\frac{\log_2(r/7)}{4}}+8r(r/14)^{\frac{\log_2(r/14)}{4}}\le r^{\frac{\log_2(r)}{4}}$,
\item\label{bounds:technical2}$8r(r/28)^{\frac{\log_2(r/28)}{4}}+6r(r/21)^{\frac{\log_2(r/21)}{4}}+4r(r/14)^{\frac{\log_2(r/14)}{4}}\le r^{\frac{\log_2(r)}{4}}$,
\item\label{bounds:technical3}$(2r+10)(r/10)^{\frac{\log_2(r/10)}{4}}+(r/5+5)(r/5)^{\frac{\log_2(r/5)}{4}}\le r^{\frac{\log_2(r)}{4}}$,
\item\label{bounds:technical4}$2p(r/p)^{\frac{\log_2(r/p)}{4}}\le r^{\log_2(r)/4}$, for every prime $p\ge 3$ with $r\ge 2p$.
\end{enumerate}
\end{lemma}

We postpone the proof of Lemma~\ref{bounds:technical} to Section~\ref{bounds:technicalsec}.

\begin{proof}[Proof of Theorem~$\ref{thrm:main}$ when $\ell=3$]
From Section~\ref{magmacomputations}, we have $r=|R|>2\,000$.

From Lemma~\ref{simple}, we see that $R$ cannot have a composition factor isomorphic to $\mathrm{PSL}_2(3
^{3^a})$, for some $a\ge 1$. Therefore,~\eqref{finall} becomes
\begin{equation}\label{finalll}
	2\left(1-\frac{\log 2}{\log r}\right)\le\frac{\log c(2)}{\log r}+\sum_{i=1}^3\left(\frac{a_i\log p_i}{4\log 2}-\frac{\log S(p_i,a_i)}{\log r}\right).
\end{equation}
Moreover, from Lemma~\ref{simple}, we have $p_2=3$ and  $p_3\in \{5,7,13,17\}$.

Let $f(r)$ be the function defined in~\eqref{function:f}. Suppose $p_3=13$ and observe that from Lemma~\ref{simple}, $|\mathrm{PSL}_3(3)|=2^4\cdot 3^3\cdot 13$ divides $r$. We have 
$$f(2^4\cdot 3^3\cdot 13)>0.46\ge 0.$$
In particular, from Lemmas~\ref{direction1} and Lemma~\ref{direction3}, we deduce  $r$ is good.  

Suppose $p_3=17$. From Lemma~\ref{simple}, $|\mathrm{PSL}_2(17)|=2^4\cdot 3^2\cdot 17$ divides $r$. We have 
$$f(2^4\cdot 3^2\cdot 17)>0.3\ge 0.$$ 
In particular, from Lemmas~\ref{direction1} and Lemma~\ref{direction3}, we deduce  $r$ is good.  

\smallskip

\noindent\textsc{Case $p_3=7$.}
We have 
\begin{align*}
f(2^2\cdot 3\cdot 7^3)&>0.4\ge 0,\\
f(2^2\cdot 3^4\cdot 7)&>0.1\ge 0,\\
f(2^3\cdot 3\cdot 7^2)&>0.1\ge 0,\\
f(2^3\cdot 3^3\cdot 7)&>0.07\ge 0,\\
f(2^3\cdot 3^2\cdot 7^2)&>0.4\ge 0,\\
f(2^6\cdot 3^2\cdot 7)&>0.02\ge 0.
\end{align*}
From Lemmas~\ref{direction1} and Lemma~\ref{direction3} and~\ref{simple}, we deduce that  $r$ is good, as long as $r$ is divisible by an element of $$\{2^2\cdot 3\cdot 7^3,2^2\cdot 3^4\cdot 7,2^3\cdot 3\cdot 7^2,2^3\cdot 3^3\cdot 7,2^3\cdot 3^2\cdot 7^2,2^6\cdot 3^2\cdot 7\}.$$ Therefore, Lemma~\ref{simple} shows that~\eqref{finalll} is satisfied, except when $r=2^{a_1}\cdot 3\cdot 7
$ or $$r\in \{252=2^2\cdot 3^2\cdot 7,756=2^2\cdot 3^3\cdot 7, 504=2^3\cdot 3^2\cdot 7,1\,008=2^4\cdot 3^2\cdot 7,2\,016=2^5\cdot 3^2\cdot 7\}.$$ 

\smallskip

As $r=|R|>2\,000$, we may exclude the cases $r\in \{252,756, 504,1\,008\}$. Assume $|R|=2\,016=2^5\cdot 3^2\cdot 7$. From Lemma~\ref{simple}, we see that $R$ has a unique non-abelian chief factor, which is isomorphic to either $\mathrm{PSL}_2(7)$ or to $\mathrm{PSL}_2(8)$.  Suppose $\mathrm{PSL}_2(8)$ is a chief factor $X/Y$ of $R$, that is, $X,Y\unlhd R$, $X\ge Y$ and $X/Y\cong\mathrm{PSL}_2(8)$. Let $$C:={\bf C}_R(X/Y).$$ By definition, $C$ is the kernel of the action $R\to\mathrm{Aut}(X/Y)$ of $R$ by conjugation on $X/Y$. In particular, as $\mathrm{Aut}(\mathrm{PSL}_2(8))\cong\mathrm{PSL}_2(8):3$ and as $3|\mathrm{PSL}_2(8)|$ does not divide $|R|$, we deduce $R/C\cong\mathrm{PSL}_2(8)$ and $|C|=2\,016/|\mathrm{PSL}_2(8)|=4$. Since $\mathrm{PSL}_2(8)$ has trivial Schur multiplier (see~\cite[page~6]{atlas}), we deduce that $R$ splits over $C$ and hence $$R\cong \mathrm{PSL}_2(8)\times C_4\hbox{ or }R\cong \mathrm{PSL}_2(8)\times C_2\times C_2.$$ Now, we have checked the veracity of the statement for these two groups with a computer. We postpone the case that $R$ has a composition factor isomorphic to $\mathrm{PSL}_2(7)$ for later. 
%

\smallskip

Assume $r=2^{a_1}\cdot 3\cdot 7$, or $r=2\, 016$ and $\mathrm{PSL}_2(7)$ is a composition factor of $R$. Thus $a_2=1$ in the first case and $a_2=2$ in the second case. In particular, by Lemma~\ref{simple}, $R$ has a unique non-abelian chief factor $X/Y$ and $X/Y\cong\mathrm{PSL}_2(7)$. Let $C:={\bf C}_R(X/Y)$. By definition, $C$ is the kernel of the action $R\to\mathrm{Aut}(X/Y)$ of $R$ by conjugation on $X/Y$. In particular, as $\mathrm{Aut}(\mathrm{PSL}_2(7))\cong\mathrm{PGL}_2(7)$, we deduce $R/C\cong\mathrm{PSL}_2(7)$ or $R/C\cong\mathrm{PGL}_2(7)$. Assume first $R/C\cong\mathrm{PSL}_2(7)$. Now, $\mathrm{PSL}_2(7)$ has four conjugacy classes of maximal $7'$-subgroups: two of these subgroups have order $24$ and two of these subgroups have order $12$. Therefore, $R$ has four conjugacy classes (with representatives $K_1,K_2,K_3$ and $K_4$, say) of $7'$-subgroups and $|K_1|=|K_2|=2^{a_1}\cdot 3^{a_2}$, $|K_3|=|K_4|=2^{a_1-1}\cdot 3^{a_2}$. Therefore, repeating the same argument we have used for solvable groups (but taking in account that $R$ has four conjugacy classes of maximal $7'$-subgroups), we deduce
\begin{align*}
|\mathrm{Sub}(R)|&\le |\mathrm{Sub}(K_1)|\frac{|R|}{|K_1|}S(7,1)\frac{|R|}{7}+|\mathrm{Sub}(K_2)|\frac{|R|}{|K_2|}S(7,1)\frac{|R|}{7}+
|\mathrm{Sub}(K_3)|\frac{|R|}{|K_3|}S(7,1)\frac{|R|}{7}+|\mathrm{Sub}(K_4)|\frac{|R|}{|K_4|}S(7,1)\frac{|R|}{7}
\\
&=2r(|\mathrm{Sub}(K_1)|+|\mathrm{Sub}(K_2)|)+
4r(|\mathrm{Sub}(K_3)|+|\mathrm{Sub}(K_4)|)\\
&
< 4rc(2)(r/7)^{\frac{\log_2(r/7)}{4}}
+8rc(2)(r/14)^{\frac{\log_2(r/14)}{4}},
\end{align*}
where the last inequality follows from Theorem~\ref{SolvableTheorem} and from the fact that $K_1,K_2,K_3$ and $K_4$ are solvable. Hence, in this case, Theorem~\ref{thrm:main} follows from Lemma~\ref{bounds:technical}~\eqref{bounds:technical1}.  Assume next $R/C\cong\mathrm{PGL}_2(7)$. Observe that $a_1\ge 4$, because a Sylow $2$-subgroup of $\mathrm{PGL}_2(7)$ has order $16.$ Now, $\mathrm{PGL}_2(7)$ has three conjugacy classes of maximal $7'$-subgroups and  these groups have order $12,16$ and $24$. Therefore, $R$ has three conjugacy classes (with representatives $K_1,K_2$ and $K_3$, say) of $7'$-subgroups and $|K_1|=2^{a_1-2}\cdot 3^{a_2}$, $|K_2|=2^{a_1}\cdot 3^{a_2-1}$ and $|K_3|=2^{a_1-1}\cdot 3^{a_2}$. Therefore,  arguing as above, we deduce
\begin{align*}
|\mathrm{Sub}(R)|&\le |\mathrm{Sub}(K_1)|\frac{|R|}{|K_1|}S(7,1)\frac{|R|}{7}+
|\mathrm{Sub}(K_2)|\frac{|R|}{|K_2|}S(7,1)\frac{|R|}{7}+
|\mathrm{Sub}(K_3)|\frac{|R|}{|K_3|}S(7,1)\frac{|R|}{7}\\
&=|\mathrm{Sub}(K_1)|S(7,1)4r+
|\mathrm{Sub}(K_2)|S(7,1)3r+
|\mathrm{Sub}(K_3)|S(7,1)2r\\
&=8rc(2)(r/28)^{\frac{\log_2(r/28)}{4}}+
6rc(2)(r/21)^{\frac{\log_2(r/21)}{4}}+
4rc(2)(r/14)^{\frac{\log_2(r/14)}{4}}.
\end{align*}
 Hence, in this case, Theorem~\ref{thrm:main} follows from Lemma~\ref{bounds:technical}~\eqref{bounds:technical2}.

\smallskip

\noindent\textsc{Case $p_3=5$.} We have
\begin{align*}
f(2^2\cdot 3\cdot 5^3)&>0.12\ge 0,\\
f(2^2\cdot 3^2\cdot 5^2)&>0.04\ge 0,\\
f(2^2\cdot 3^5\cdot 5)&>0.17\ge 0,\\
f(2^3\cdot 3^4\cdot 5)&>0.15\ge 0,\\
f(2^4\cdot 3^3\cdot 5)&>0.04\ge 0,\\
f(2^5\cdot 3\cdot 5^2)&>0.03\ge 0,\\
f(2^9\cdot 3^2\cdot 5)&>0.04\ge 0.
\end{align*}
From Lemmas~\ref{direction1} and Lemma~\ref{direction3} and~\ref{simple}, we deduce that $f(r)\ge 0$ and hence $r$ is good, as long as $r$ is divisible by an element of $$\{2^2\cdot 3\cdot 5^3,2^2\cdot 3^2\cdot 5^2,2^2\cdot 3^5\cdot 5,2^3\cdot 3^4\cdot 5, 2^4\cdot 3^3\cdot 5, 2^5\cdot 3\cdot 5^2,2^9\cdot 3^2\cdot 5\}.$$ Therefore, Lemma~\ref{simple} shows that~\eqref{finalll} is satisfied, except when $r=2^{a_1}\cdot 3\cdot 5
$ or 
\begin{align*}
r\in \{&300=2^2\cdot 3\cdot 5^2,600=2^3\cdot 3\cdot 5^2, 1\,200=2^4\cdot 3\cdot 5^2,\\
&1\,620=2^2\cdot 3^4\cdot 5, 540=2^2\cdot 3^3\cdot 5, 1\,080=2^3\cdot 3^3\cdot 5,\\
&180=2^2\cdot 3^2\cdot 5,
360=2^3\cdot 3^2\cdot 5,
720=2^4\cdot 3^2\cdot 5,
1\,440=2^5\cdot 3^2\cdot 5,\\
&2\,880=2^6\cdot 3^2\cdot 5,
5\,760=2^7\cdot 3^2\cdot 5,
11\,520=2^8\cdot 3^2\cdot 5\}.
\end{align*}

\smallskip

As $r=|R|>2,000$, we may exclude the cases $r\in \{180, 300, 360, 540, 600, 720, 1\,080, 1\,200, 1\,440, 1\,620\}$. Assume $r\in\{2\,880,5\,760,11\,520\}$. Thus $r=2^{a_1}\cdot 3^2\cdot 5$, with $a_1\in \{6,7,8\}$. From Lemma~\ref{simple}, we see that $R$ has a unique non-abelian composition factors, which is isomorphic to either $\mathrm{Alt}(5)$ or to $\mathrm{Alt}(6)$.  Suppose $\mathrm{Alt}(6)$ is a composition factor $X/Y$ of $R$. Let $C:={\bf C}_R(X/Y)$. By definition, $C$ is the kernel of the action $R\to\mathrm{Aut}(X/Y)$ of $R$ by conjugation on $X/Y$. In particular, as $\mathrm{Aut}(\mathrm{Alt}(6))\cong\mathrm{P}\Gamma\mathrm{L}_2(9)$, we deduce 
\begin{itemize}
\item $R/C\cong\mathrm{Alt}(6)$ and $|C|=2^{a_1-3}$, or
\item $R/C\cong\mathrm{PGL}_2(9)$ and $|C|=2^{a_1-4}$, or
\item $R/C\cong M_{10}$ and $|C|=2^{a_1-4}$, or 
\item $R/C\cong\mathrm{P}\Sigma\mathrm{L}_2(9)$ and $|C|=2^{a_1-4}$, or
\item $R/C\cong\mathrm{P}\Gamma\mathrm{L}_2(9)$ and $|C|=2^{a_1-5}$.
\end{itemize}
Since $\mathrm{Alt}(6)$ has Schur multiplier of order $6$ (see~\cite[page~6]{atlas}) and since $C$ is a $2$-group, we deduce that either  $R$ splits over $C$ and hence $R\cong R/C\times C$, or $R\cong R/C\circ C$. Using this information we recover the various isomorphism classes of $R$ and check the veracity of Theorem~\ref{thrm:main} in each case. For instance, when $R/C\cong\mathrm{Alt}(6)$, $R\cong R/C\times C$ and $a_1=6$, we have that $R$ is isomorphic to one of the following five groups
$$\mathrm{Alt}(6)\times C_2^3,\, \mathrm{Alt}(6)\times C_2\times C_4,\,\mathrm{Alt}(6)\times C_8,\,\mathrm{Alt}(6)\times D_4,\,\mathrm{Alt}(6)\times Q_8.$$

\smallskip

Assume $r=2^{a_1}\cdot 3\cdot 5$, or $r\in \{2\,880,5\,760,11\,520\}$ and $\mathrm{Alt}(5)$ is a composition factor of $R$. Thus $a_2=1$ in the first case and $a_2=2$ in the second case. In particular, by Lemma~\ref{simple}, $R$ has a unique non-abelian chief factor $X/Y$ and $X/Y\cong\mathrm{Alt}(5)$. Let $C:={\bf C}_R(X/Y)$. By definition, $C$ is the kernel of the action $R\to\mathrm{Aut}(X/Y)$ of $R$ by conjugation on $X/Y$. In particular, as $\mathrm{Aut}(\mathrm{Alt}(5))\cong\mathrm{Sym}(5)$, we deduce $R/C\cong\mathrm{Alt}(5)$ or $R/C\cong\mathrm{Sym}(5)$. Here we need to argue slightly differently from the case in the previous paragraph, because otherwise we end up with too many cases to be checked with a computer. Now $\mathrm{Alt}(5)$ and $\mathrm{Sym}(5)$ both have two conjugacy classes of maximal $5'$-subgroups and these subgroups have order $6$ and $12$ in $\mathrm{Alt}(5)$ and have order $12$ and $24$ in $\mathrm{Sym}(5)$. Therefore, $R$ has two conjugacy classes (with representatives $K_1$ and $K_2$, say) of $5'$-subgroups and $|K_1|=2^{a_1-1}\cdot 3^{a_2}$ and  $|K_2|=2^{a_1}\cdot 3^{a_2}$. Let us call $a$ the number of subgroups of $R$ having order relatively prime to $5$ and let us call $b$ the number of subgroups of $R$ having order divisible by $5$. Since $K_1$ has $10$ conjugates in $R$ and since $K_2$ has $5$ conjugates in $R$, we deduce
\begin{align}\label{eq:A}
a&\le 10|\mathrm{Sub}(K_1)|+5|\mathrm{Sub}(K_2)|\le 10c(2)(r/10)^{\frac{\log_2(r/10)}{4}}+5c(2)(r/5)^{\frac{\log_2(r/5)}{4}}.
\end{align}
Now, let $H$ be an arbitrary subgroup of $R$ having order divisible by $5$. Then $H=\langle A,B\rangle$, where $A$ is a $5'$-subgroup of $H$ and $B$ is a Sylow $5$-subgroup of $H$. Now, $A$ is contained in one of the $15$ maximal $5'$-subgroups of $R$. However, since $B$ acts transitively on the five conjugates of $K_2$, we may assume (replacing $A$ with a suitable $B$-conjugate if necessary) that either $A$ is contained in one of the $10$ conjugates of $K_1$ or $A\le K_2$. Taking this in account, we have
\begin{align}\label{eq:B}
b&\le 10|\mathrm{Sub}(K_1)|\frac{r}{5}+|\mathrm{Sub}(K_2)|\frac{r}{5}\le c(2)\cdot 2r\cdot(r/10)^{\frac{\log_2(r/10)}{4}}+c(2)\cdot\frac{r}{5}\cdot(r/5)^{\frac{\log_2(r/5)}{4}}.
\end{align}
Clearly, $|\mathrm{Sub}(R)|=a+b$. Using~\eqref{eq:A} and~\eqref{eq:B}, we obtain $a+b\le c(2)r^{\log_2r/4}$, except when $a_1=2$. Therefore, in this case, Theorem~\ref{thrm:main} immediately follows from Lemma~\ref{bounds:technical}~\eqref{bounds:technical3}.
\end{proof}

\section{The final cases}\label{sec:finalcases}
Before dealing with the remaining cases, we need three general results.

\begin{lemma}\label{lemma:schur}
Let $R$ be a counterexample of minimal order to Theorem~$\ref{thrm:main}$. Then $R$ has no normal non-identity Sylow $p$-subgroup.
\end{lemma}
\begin{proof}
We argue by contradiction and we let $R$ be a counterexample of minimal order to Theorem~\ref{thrm:main} admitting a normal non-identity Sylow $p$-subgroup $P$.  From the Schur-Zassenhaus theorem, let $K$ be a complement of $P$ in $R$.  Thus $p\ge 3$.

 If $p=2$, then from the Odd Order Theorem $P$ and $R/P$ are solvable, and hence so is $R$. However, this contradicts Theorem~\ref{SolvableTheorem}.

From Section~\ref{magmacomputations}, we may suppose that $|R|>2\,000$ and, from Section~\ref{sec:l=3}, we may suppose that $\ell\ge 4$.

Let $H$ be a subgroup of $R$. Then, from the Schur-Zassenhaus theorem, $H=H_{p'}H_{p}$, where $H_{p'}$ is a Hall $p'$-subgroup of $H$ and $H_{p}$ is a Sylow $p$-subgroup of $H$. Now, from the Schur-Zassenhaus theorem, $H_{p'}$ is conjugate, via an element of $P$, to a subgroup of $K$ and, from Sylow's theorem, $H_{p}$ is conjugate, via an element of $K$, to a subgroup of $P$. In particular, we have at most $1+(|\mathrm{Sub}(K)|-1)|P|$ choices for $H_{p'}$, because every non-identity subgroup of $K$ as at most $|P|$ conjugates. Similarly, we have at most $1+(|\mathrm{Sub}(P)|-1)|K|$ choices for $H_{p}$, because every non-identity subgroup of $P$ as at most $|K|=|R|/|P|$ conjugates. Write $|P|=p^a$ and $r=|R|$. Therefore,
\begin{equation}\label{eq:silly}
|\mathrm{Sub}(R)|\le (1+(|\mathrm{Sub}(K)|-1)p^a)\cdot \left(
1+(|\mathrm{Sub}(P)|-1)\frac{r}{p^{a}}\right).
\end{equation}
Using Lemma~\ref{lemma:new}, this can be simplified in
$$|\mathrm{Sub}(R)|\le |\mathrm{Sub}(K)|p^{a}\cdot |\mathrm{Sub}(P)|\frac{r}{p^{a}}\le |\mathrm{Sub}(K)|S(p,a)r.$$
As $K<R$, $K$ is not a counterexample to Theorem~\ref{thrm:main} and hence
\begin{align*}
|\mathrm{Sub}(R)|&< c(2)(r/p^{a})^{\frac{\log_2(r/p^{a})}{4}}S(p,a)r=
c(2)r^{\frac{\log_2r}{4}}\cdot\frac{S(p,a)r}{p^{\frac{a\log_2(r/p^{a})}{4}}r^{\frac{\log_2(p^{a})}{4}}}\\
&=c(2)r^{\frac{\log_2r}{4}}\cdot\frac{S(p,a)r}{(r^2/p^a)^{\frac{\log_2(p^{a})}{4}}}.
\end{align*}
As $R$ is a counterexample to Theorem~\ref{thrm:main}, we have
$$
\frac{S(p,a)r}{(r^2/p^a)^{\frac{\log_2(p^{a})}{4}}}>1.$$
From Proposition~\ref{proposition:appendix}, we obtain that one of~\eqref{appendix0}--\eqref{appendix5} is satisfied. As $r>2\,000$, we deduce $p\in \{3,5\}$ and $a=1$.

When $a=1$,  $1$ and $P$ are the only $p$-subgroups of $R$ and hence we may refine~\eqref{eq:silly} with
\begin{equation*}
|\mathrm{Sub}(R)|\le (1+(|\mathrm{Sub}(K)|-1)p^a)\cdot 2\le 2p|\mathrm{Sub}(K)|< 2pc(2)(r/p)^{\frac{\log_2(r/p)}{4}}.
\end{equation*}
Now, the proof follows from Lemma~\ref{bounds:technical}~\eqref{bounds:technical4}. 
\end{proof}

\begin{lemma}\label{lemma:schur2}
Let $C$ be a solvable group and let $p$ be a prime divisor of $|C|$ with the property that, for each prime power divisor $q$ of $|C|$ with $q>1$ and $p\nmid q$, $p$ is relatively prime to $q-1$. Then $R$ has a normal  Sylow $p$-subgroup.
\end{lemma}
\begin{proof}
We argue by contradiction and we let $C$ be a counterexample of minimal order. Let $N$ be a minimal normal subgroup of $C$. As $C$ is solvable, $N$ has order a prime power $q>1$. We adopt the ``bar'' notation for $\bar C=C/N$. If $p$ is relatively prime to $|\bar C|$, then $N$ is a normal Sylow $p$-subgroup of $C$, contradicting the fact that $C$ is a counterexample to the statement of this lemma. Thus $p\mid |\bar C|.$ As $|\bar C|<|C|$, $\bar C$ has a normal Sylow $p$-subgroup $\bar P$. Thus $\bar P=NP/N$, where $P$ is a Sylow $p$-subgroup of $C$. As $\bar P\unlhd \bar C$, we have $NP\unlhd C$ and hence the minimality of $C$ gives $C=NP$. Now the action of $P$ by conjugation on $N$ endows $N$ of the structure of  module for $P$. As $N$ is a minimal normal subgroup of $C$, $P$ acts irreducibly on $N$ and hence $p$ divides $|N|-1=q-1$, which is a contradiction.
\end{proof}

\subsection{An algorithm: step~1}\label{sec:algorithm}
In view of Proposition~\ref{lmax}, there are only a finite number of counterexamples to Theorem~\ref{thrm:main} and our task is to show that actually there are no counterexamples. We now describe an algorithm that greatly reduces the number of exceptions we need to analyze in detail. 

The first step in our algorithm is to refine even further the functions~\eqref{function:f} and~\eqref{function:ffttt}.
The input in the first step of our algorithm is a positive integer $r$ satisfying none of the conditions in Proposition~\ref{lmax}. The output of our algorithm is ``yes'' if a finite group of order $r$ satisfies Theorem~\ref{thrm:main} and is ``unknown'' if our procedure cannot exclude the existence of a counterexample to Theorem~\ref{thrm:main}.

First, as usual, we write $r=p_1^{a_1}\cdots p_\ell^{a_\ell}$, where $p_1<\cdots <p_\ell$. Next, for each $i\in \{1,\ldots,\ell\}$, 
\begin{itemize}
\item when $p_i=2$, we let $n_i=r/p_i^{a_i}$,
\item when $p_i=3$ and $r$ is divisible by the order of $\mathrm{PSL}_2(3^{3^f})$ for some $f\ge 0$, we let $n_i$ be the largest divisor of $r/p_i^{a_i}$ with $n_i\equiv 1\mod p_i$ and with $n_i\le r/p_i^{a_i}$, otherwise
\item we let $n_i$ be the largest divisor of $r/p_i^{a_i}$ with $n_i\equiv 1\mod p_i$ and with $n_i< r/p_i^{a_i}$.
\end{itemize} 
Observe that, by~\cite{GMN}, $n_i$ is an upper bound on the number of Sylow $p_i$-subgroups of a non-solvable finite group of order $r$.  From Section~\ref{sec:ao}, for every non-solvable group $R$ of order $r$, we have
\begin{equation}\label{betterbound:1}|\mathrm{Sub}(R)|\le \prod_{i=1}^\ell n_iS(p_i,a_i).\end{equation}
We have computed $n_1,\ldots,n_\ell$ and we have computed the right hand side of this inequality. When $n_i=1$ for some $i$ or when this number is less than 
 $c(2)r^{\log_2r/4}$, we return ``yes'' otherwise we return ``unknown''. The list of positive integers $r$ where this procedure returns ``unknown'' is reported in 
 Table~\ref{tableSimpleSimple}. From Lemma~\ref{lemma:schur} and from~\eqref{betterbound:1}, when the procedure returns ``yes'', finite groups of order $r$ satisfy Theorem~\ref{thrm:main}. Therefore, the cardinalities that require further considerations are in Table~\ref{tableSimpleSimple}.
 
 Using the order of the non-abelian simple groups (via the Classification of Finite Simple Groups), we are able to verify two important facts in our second procedure.

\begin{table}[!ht]
\begin{tabular}{l|l}
$\ell$&$r$\\\hline
$10$&$12939386460$\\
&$13831757940$\\
&$26324958660$\\\hline
$9$&$446185740$\\
&$892371480$\\
&$1784742960$\\
&$562582020$\\
&$1125164040$\\
&$601380780$\\
&$1804142340$\\
&$1202761560$\\
&$717777060$\\
&$1435554120$\\
&$795374580$\\
&$1028167140$\\
&$2289126840$\\
&$681020340$\\
&$1362040680$\\
&$868888020$\\
&$962821860$\\
&$1103722620$\\
&$917896980$\\
&$1835793960$\\
&$761140380$\\
&$1522280760$\\
&$813632820$\\
&$971110140$\\
&$1076095020$\\
&$1025884860$
\end{tabular}
\begin{tabular}{l|l}
$\ell$&$r$\\
$8$&
1
19399380
2
135795660
3
96996900
4
58198140
5
38798760
6
116396280
7
77597520
8
23483460
9
70450380
10
46966920
11
93933840
12
29609580
13
59219160
14
118438320
15
31651620
16
94954860
17
63303240
18
37777740
19
75555480
20
83723640
21
43903860
22
87807720
23
47987940
24
54114060
25
108228120
26
62282220
27
68408340
28
74534460
29
26246220
30
78738660
31
52492440
32
104984880
33
33093060
34
35375340
35
70750680
36
46786740
37
93573480
38
49069020
39
98138040
40
60480420
41
67327260
42
69609540
43
40060020
44
160240080
45
42822780
46
51111060
47
59399340
48
64924860
49
129849720
50
114654540
51
53993940
52
64444380
53
71411340
54
205525320
55
106246140

\end{tabular}
\begin{tabular}{l|l}
$\ell$&$r$\\\hline
$4$
&
420, 2940, 2100, 1260, 11340, 840, 4200, 2520, 7560, 1680, 11760, 5040, 3360, 6720, 20160, 13440, 26880, 
53760, 107520, 215040, 430080, 860160, 1720320, 6881280, 13762560, 55050240, 440401920, 660, 3300, 1980, 
1320, 3960, 2640, 10560, 21120, 84480, 1560, 7800, 3120, 4080, 8160, 1140, 4560, 2760, 1740, 924, 1848, 3696,
14784, 1092, 2184, 3444, 1716
\\\hline
\end{tabular}
\caption{Positive integers returning ``unknown'' in the step~1 of the algorithm}\label{tableSimpleSimple}
\end{table}

l=5

1 4620
2 32340
3 23100
4 13860
5 41580
6 9240
7 27720
8 83160
9 18480
10 55440
11 36960
12 73920
13 147840
14 295680
15 591360
16 5460
17 38220
18 27300
19 16380
20 49140
21 10920
22 54600
23 32760
24 21840
25 43680
26 131040
27 87360
28 7140
29 21420
30 64260
31 14280
32 42840
33 28560
34 142800
35 57120
36 114240
37 7980
38 23940
39 15960
40 31920
41 95760
42 255360
43 9660
44 48300
45 19320
46 77280
47 154560
48 12180
49 24360
50 26040
51 46620
52 17220
53 72240
54 8580
55 17160
56 68640
57 22440
58 44880
59 45540
60 19140
61 31020
62 13260
63 29640
64 28860
65 12012
66 36036
67 24024
68 48048
69 192192
70 62832
71 17556
72 42504
73 44268

\smallskip

\noindent\textsc{Fact 1: } Let $T$ be a non-abelian  simple group  whose order divides $r$. Then $T$ appears in Table~\ref{tableSimple}. Observe that the order of the outer automorphism group of $T$ is not divisible by primes larger than $3$. 

\begin{table}[!ht]
\begin{tabular}{c|c|c}
Group&Order &Order of outer Automorphism\\\hline
$\mathrm{Alt}(5)$&$60=2^2\cdot 3\cdot 5$&2\\
$\mathrm{PSL}_2(7)$&$168=2^3\cdot 3\cdot 7$&2\\
$\mathrm{Alt}(6)$&$360=2^3\cdot 3^2\cdot 5$&4\\
$\mathrm{PSL}_2(8)$&$504=2^3\cdot 3^2\cdot 7$&3\\
$\mathrm{PSL}_2(11)$&$660=2^2\cdot 3\cdot 5\cdot 11$&2\\
$\mathrm{PSL}_2(13)$&$1\,092=2^2\cdot 3\cdot 7\cdot 13$&2\\
$\mathrm{PSL}_2(17)$&$2\, 448=2^4\cdot 3^2\cdot 17$&2\\
$\mathrm{Alt}(7)$&$2\,520=2^3\cdot 3^2\cdot 5\cdot 7$&2\\
$\mathrm{PSL}_2(19)$&$3\,420=2^2\cdot 3^2\cdot 5\cdot 19$&2\\

$\mathrm{PSL}_2(16)$&$4\,080=2^4\cdot 3\cdot 5\cdot 17$&4\\
$\mathrm{PSL}_3(3)$&$5\,616=2^4\cdot 3^3\cdot 13$&2\\
$\mathrm{PSL}_2(23)$&$6\,072=2^3\cdot 3\cdot 11\cdot 23$&2\\
$\mathrm{PSL}_2(25)$&$7\,800=2^3\cdot 3\cdot 5^2\cdot 13$&4\\
$M_{11}$&$7\,920=2^4\cdot 3^2\cdot 5\cdot 11$&1\\
$\mathrm{PSL}_{2}(27)$&$9\,828=2^2\cdot 3^3\cdot 7\cdot 13$&6\\

$\mathrm{PSL}_{2}(29)$&$12\,180=2^2\cdot 3\cdot 5\cdot 7\cdot 29$&2\\

$\mathrm{Alt}(8)$&$20\,160=2^6\cdot 3^2\cdot 5\cdot 7$&2\\
$\mathrm{PSL}_{3}(4)$&$20\,160=2^6\cdot 3^2\cdot 5\cdot 7$&12\\
$\mathrm{Suz}(8)$&$29\,120=2^6\cdot 5\cdot 7\cdot 13$&3\\

$\mathrm{PSL}_{2}(41)$&$34\,440=2^3\cdot 3\cdot 5\cdot 7\cdot 41$&2\\

$\mathrm{PSL}_2(43)$&$39\,732=2^2\cdot 3\cdot 7\cdot 11\cdot 43$&2\\
$\mathrm{PSL}_2(67)$&$150\,348=2^2\cdot 3\cdot 11\cdot 17\cdot 67$&2\\
$J_1$&$175\, 560=2^3\cdot 3\cdot 5\cdot 7\cdot 11\cdot 19$&1\\

\end{tabular}
\caption{Putative non-abelian simple sections in a counterexample to Theorem~\ref{thrm:main}.}\label{tableSimple}
\end{table}

\noindent\textsc{Fact 2: } Let $\kappa\ge 2$ and let $T_1,\ldots,T_\kappa$ be non-abelian simple groups with $r$ divisible by $|T_1|\cdot |T_2|\cdots |T_\kappa|$. Then $\kappa=2$ and $T_1\times T_2$ is isomorphic to either $$
\mathrm{Alt}(5)\times\mathrm{PSL}_2(7), \hbox{ or } 
\mathrm{Alt}(5)\times\mathrm{PSL}_2(13), 
\hbox{ or } \mathrm{PSL}_2(11)\times\mathrm{PSL}_2(13) 
.$$
 In particular, a non-solvable group $R$ of order $r$ (where the previous procedure has returned ``unknown'') has at most two non-abelian simple sections, and if $R$ has two non-abelian simple sections, then these two sections are not isomorphic. 

\subsection{An algorithm: step~2}

We are now ready to describe our second procedure that needs to be applied to all cases where the previous procedure returns ``unknown''.

The input of the second procedure is a positive integer $r$. First, we determine the set $\mathcal{D}$ all the divisors $d$ of $r$, with the property that $d$ is the order of a direct product of non-abelian simple groups. (The number $d$ represents the cardinality of the product of the cardinalities of the non-abelian simple sections in a non-solvable group of order $r$.) Clearly, in this step, we may use the information in Table~\ref{tableSimple}.

In the case that there is no such divisor $d$, that is $\mathcal{D}=\emptyset$, we stop our algorithm and we return ``yes'': this represents the fact that a finite group of order $r$ is solvable because $r$ is not divisible by the order of a non-abelian simple group. 

Let $\kappa$ be the number of non-abelian factors in a composition series of a non-solvable group $R$ of order $r$. From Fact~2, $\kappa\le 2$ and, in the case $\kappa=2$, the two non-abelian factors are non-isomorphic.  Let $X_1/Y_1,\ldots,X_{\kappa}/Y_\kappa$  be the non-abelian simple sections of $R$ and let $d=|X_1/Y_1|\cdots |X_\kappa/Y_\kappa|$. When $\kappa=1$, by taking a suitable chief series, we may suppose that $X_1$ and $Y_1$ are normal subgroups of $R$. Similarly, when $\kappa=2$, as $X_1/Y_1\not\cong X_2/Y_2$, by taking a suitable chief series, we may suppose that $X_1,X_2,Y_1,Y_2$ are normal subgroups of $R$. Let $C={\bf C}_R(X_1/Y_1)$ when $\kappa=1$ and let $C={\bf C}_R(X_1/Y_1)\cap{\bf C}_R(X_2/Y_2)$ when $\kappa=2$. Now, $C$ is the kernel of the action of $R$ by conjugation on $X_1/Y_1$ (when $\kappa=1$) and on $X_1/Y_1\times X_2/Y_2$ (when $\kappa=2$). 
When $\kappa=1$, $R/C$ is almost simple with socle $X_1/Y_1$ and $C$ is solvable. When $\kappa=2$, $R/C$ is isomorphic to a subgroup of $\mathrm{Aut}(X_1/Y_1)\times\mathrm{Aut}(X_2/Y_2)$ and $C$ is solvable.

 The order of $C$ divides $r/d$. Now, we select all prime divisor $p\ge 5$ of $r/d$ with 
\begin{itemize}
\item $p\nmid d$, and
\item for each divisor $q$ of $r/d$ with $q>1$ and with $q$ a prime power, $p$ is relatively prime to $q-1$.
\end{itemize}
If there is at least one such prime, we stop the computations and we return ``yes''; indeed, by Lemma~\ref{lemma:schur} and~\ref{lemma:schur2}, $r$ is not the order of a minimal counterexample to Theorem~\ref{thrm:main}.

The only positive integers $r$ where this produce did not return ``yes'' are recorded in Table~\ref{tableexceptions1},~\ref{tableexceptions2} and~\ref{tableexceptions3}. A direct inspection on these tables shows that $\ell\in \{4,5,6\}$. Moreover, except when $|R|=20\,160$ and $\mathrm{Alt}(5)\times\mathrm{PSL}_2(7)$ is a section of $R$, $R$ has a unique non-abelian composition factor. Suppose that $|R|=20\,160$ and $\mathrm{Alt}(5)\times \mathrm{PSL}_2(7)$ is a section of $R$. As $|\mathrm{Alt}(5)\times\mathrm{PSL}_2(7)|=10\,080$, we deduce that $R$ has a normal subgroup $N$ such that either
\begin{itemize}
\item $|R:N|=2$ and $N\cong\mathrm{Alt}(5)\times \mathrm{PSL}_2(7)$, or
\item $|N|=2$ and $R/N\cong\mathrm{Alt}(5)\times\mathrm{PSL}_2(7)$.
\end{itemize}
Taking in account that the Schur multiplier and the outer automorphism group of $\mathrm{Alt}(5)$ and $\mathrm{PSL}_2(7)$ have order $2$, we deduce that $R$ is isomorphic to 
$$
\mathrm{Sym}(5)\times\mathrm{PSL}_2(7),\,
\mathrm{Alt}(5)\times\mathrm{PGL}_2(7),\,
(\mathrm{Alt}(5)\times \mathrm{PSL}_2(7)).2,
$$
or to
$$
\mathrm{SL}_2(5)\times\mathrm{PSL}_2(7),\,
\mathrm{Alt}(5)\times\mathrm{SL}_2(7),\,
\mathrm{SL}_2(5)\circ \mathrm{SL}_2(7)
.$$
The veracity of Theorem~\ref{thrm:main} for these six groups can be verified with a computer. Therefore, for the rest of our argument, we may suppose that $R$ has a unique non-abelian simple factor $T$. Therefore, $R$ has a normal solvable subgroup $C$ with $R/C$ almost simple with socle $T$.
\begin{table}[!ht]
\begin{tabular}{c|c|c}\hline
$r=|R|$&non-abelian sections&$|R|/|T|$\\\hline
$175\,560=2^3\cdot 3\cdot 5\cdot 7\cdot 11\cdot 19$&   $J_1$&1\\
$351\,120=2^4\cdot 3\cdot 5\cdot 7\cdot 11\cdot 19$& $J_1$&2\\
$702\,240=2^5\cdot 3\cdot 5\cdot 7\cdot 11\cdot 19$& $J_1$&4\\
\end{tabular}
\caption{Exceptions with $\ell=6$}
\label{tableexceptions1}
\end{table}

\begin{table}[!ht]
\begin{tabular}{c|c|c}\hline
$r=|R|$&non-abelian sections&$|R|/|T|$\\\hline
$36\,960=2^5\cdot 3\cdot 5\cdot 7\cdot 11$&$\mathrm{PSL}_2(11)$&$56$\\
$73\,920=2^6\cdot 3\cdot 5\cdot 7\cdot 11$&$\mathrm{PSL}_2(11)$&$112$\\


$147\,840=2^7\cdot 3\cdot 5\cdot 7\cdot 11$&$\mathrm{PSL}_2(11)$&$224$\\

$295\,680=2^8\cdot 3\cdot 5\cdot 7\cdot 11$&$\mathrm{PSL}_2(11)$&$448$\\
$591\,360=2^9\cdot 3\cdot 5\cdot 7\cdot 11$&$\mathrm{PSL}_2(11)$&$896$\\


$87\,360=2^6\cdot 3\cdot 5\cdot 7\cdot 13$&$\mathrm{PSL}_2(13)$&80\\



$12\,180=2^2\cdot 3\cdot 5\cdot 7\cdot 29$&$\mathrm{PSL}_2(29)$&1\\
$24\,360=2^3\cdot 3\cdot 5\cdot 7\cdot 29$&$\mathrm{PSL}_2(29)$&2\\
\end{tabular}
\caption{Exceptions with $\ell=5$}
\label{tableexceptions2}
\end{table}

\begin{table}[!ht]
\begin{tabular}{c|c|c|c}\hline
$r=|R|$&non-abelian sections&$|R|/|T|$\\\hline
$2\,520=2^3\cdot 3^2\cdot 5\cdot 7$&$\mathrm{Alt}(7)$&1\\
$5\,040=2^4\cdot 3^2\cdot 5\cdot 7$&$\mathrm{Alt}(7)$&2\\
$7\,560=2^3\cdot 3^3\cdot 5\cdot 7$&$\mathrm{Alt}(7)$ &3\\
$3\,360=2^5\cdot 3\cdot 5\cdot 7$&$\mathrm{Alt}(5)$ &56\\
$6\,720=2^6\cdot 3\cdot 5\cdot 7$&$\mathrm{Alt}(5)$&112\\
$20\,160=2^6\cdot 3^2\cdot 5\cdot 7$&$\mathrm{Alt}(5)$&336\\
$20\,160=2^6\cdot 3^2\cdot 5\cdot 7$&$\mathrm{Alt}(6)$ &56\\
$20\,160=2^6\cdot 3^2\cdot 5\cdot 7$&$\mathrm{Alt}(7)$&8\\
$20\,160=2^6\cdot 3^2\cdot 5\cdot 7$& $\mathrm{Alt}(8)$ or $\mathrm{PSL}_3(4)$&1\\
$20\,160=2^5\cdot 3^2\cdot 5\cdot 7$&$\mathrm{Alt}(5)\times\mathrm{PSL}_2(7)$&2\\
$13\,440=2^7\cdot 3\cdot 5\cdot 7$&$\mathrm{Alt}(5)$ &224\\
$13\,440=2^7\cdot 3\cdot 5\cdot 7$&$\mathrm{PSL}_2(7)$ &80\\
$26\,880=2^8\cdot 3\cdot 5\cdot 7$&$\mathrm{Alt}(5)$ &448\\
$26\,880=2^8\cdot 3\cdot 5\cdot 7$&  $\mathrm{PSL}_2(7)$&160\\
$53\,760=2^9\cdot 3\cdot 5\cdot 7$&$\mathrm{Alt}(5)$ &896\\
$53\,760=2^9\cdot 3\cdot 5\cdot 7$&$\mathrm{PSL}_2(7)$  &320\\
$107\,520=2^{10}\cdot 3\cdot 5\cdot 7$&$\mathrm{Alt}(5)$ &$1\,792$\\
$107\,520=2^{10}\cdot 3\cdot 5\cdot 7$&$\mathrm{PSL}_2(7)$ &640\\
$215\,040=2^{11}\cdot 3\cdot 5\cdot 7$&$\mathrm{Alt}(5)$ &$3\,584$\\
$215\,040=2^{11}\cdot 3\cdot 5\cdot 7$& $\mathrm{PSL}_2(7)$&$1\,280$\\
$430\,080=2^{12}\cdot 3\cdot 5\cdot 7$&$\mathrm{Alt}(5)$ &$7\,168$\\
$430\,080=2^{12}\cdot 3\cdot 5\cdot 7$&$\mathrm{PSL}_2(7)$ &$2\,560$\\
$660=2^2\cdot 3\cdot 5\cdot 11$& $\mathrm{PSL}_2(11)$&1\\
$1\,320=2^3\cdot 3\cdot 5\cdot 11$& $\mathrm{PSL}_2(11)$ &2\\
$1\,980=2^2\cdot 3^2\cdot 5\cdot 11$& $\mathrm{PSL}_2(11)$ &3\\
$2\,640=2^4\cdot 3\cdot 5\cdot 11$& $\mathrm{PSL}_2(11)$ &4\\
$3\,300=2^2\cdot 3\cdot 5^2\cdot 11$& $\mathrm{PSL}_2(11)$ &5\\
$3\,960=2^3\cdot 3^2\cdot 5\cdot 11$& $\mathrm{PSL}_2(11)$ &6\\
$10\,560=2^6\cdot 3\cdot 5\cdot 11$& $\mathrm{PSL}_2(11)$ &16\\
$21\,120=2^7\cdot 3\cdot 5\cdot 11$& $\mathrm{PSL}_2(11)$ &32\\
$84\,480=2^9\cdot 3\cdot 5\cdot 11$& $\mathrm{PSL}_2(11)$ &128\\
$7\,800=2^3\cdot 3\cdot 5^2\cdot 13$&$\mathrm{PSL}_2(25)$&1\\
$4\,080=2^4\cdot 3\cdot 5\cdot 17$& $\mathrm{PSL}_2(16)$&1\\
$8\,160=2^5\cdot 3\cdot 5\cdot 17$&$\mathrm{PSL}_2(16)$ &2\\

$1\,092=2^2\cdot 3\cdot 7\cdot 13$&$\mathrm{PSL}_2(13)$&1\\
$2\,184=2^3\cdot 3\cdot 7\cdot 13$&$\mathrm{PSL}_2(13)$&2\\
\end{tabular}
\caption{Exceptions with $\ell=4$}
\label{tableexceptions3}
\end{table}

\subsection{The case $\ell=6$}\label{sec:6}From Table~\ref{tableexceptions1}, we have $T=J_1$. Since  $J_1$ has trivial Schur multiplier and trivial outer automorphism group, we deduce $R=T\times C$. As $|C|\in \{1,2,4\}$, the veracity of Theorem~\ref{thrm:main} for these groups can be verified with a computer. 

\subsection{The case $\ell=5$}\label{sec:5}
We use the information from Table~\ref{tableexceptions2}. When $T=\mathrm{PSL}_2(29)$, we have $|C|\le 2$ and hence we deduce that $R\in \{\mathrm{PSL}_2(29),\mathrm{SL}_2(29),\mathrm{PGL}_2(29)\}$. Here, we check Theorem~\ref{thrm:main} with a computer.

Suppose now $T=\mathrm{PSL}_2(p)$, with $p\in \{11,13\}$. Thus $R/C\cong\mathrm{PSL}_2(p)$ or $R/C\cong\mathrm{PGL}_2(p)$. Observe that $p$ is relatively prime to $|C|$. We adopt the ``bar'' notation for the projection of $R$ onto $\bar R=R/C$. Let $\bar{H}_1,\ldots,\bar{H}_\kappa$ be representatives for the $\bar R$-conjugacy classes of the maximal $p'$-subgroups of $\bar R$. Then the preimages $H_1,\ldots,H_\kappa$ are representatives for the $R$-conjugacy classes of the maximal $p'$-subgroups of $R$. Since the Sylow $p$-subgroups of $R$ are cyclic of order $p$ and since $R$ has at most 
$r/p$ Sylow $p$-subgroups, we deduce
\begin{align*}
|\mathrm{Sub}(R)|&\le \sum_{i=1}^\kappa|\mathrm{Sub}(H_i)|\cdot \frac{r}{|H_i|}\cdot 2\cdot \frac{r}{p}
\le \frac{2r^2c(2)}{p} \sum_{i=1}^\kappa |H_i|^{\frac{\log_2|H_i|}{4}-1}.
\end{align*}
We have implemented this function using the information on $\bar{R}$ and in all cases this bound is less than $c(2)|R|^{\log_2|R|/4}$.
\subsection{The case $\ell=4$}\label{sec:4}
We omit the analysis of this case. All groups can be checked with arguments analogous to the methods used in Section~\ref{sec:5}.

\section{Appendix}\label{Appendix}
\subsection{Proof of Proposition~\ref{proposition:appendix}}\label{sec:appendixproposition}We consider various cases depending on the value of $a_i$.

\smallskip

\noindent\textsc{Case $a_i=1$. }We have $S(p_i,a_i)=2$. As $\ell\ge 2$, we have $r/p_i\ge p_1\cdots p_{i-1}\cdot p_{i+1}\cdots p_\ell\ge 2$. Hence, we have
\begin{align}\label{friday111}
\frac{S(p_i,a_i)r}{(r^2/p_i^{a_i})^{a_i\frac{\log_2(p_i)}{4}}}&=\frac{2r}{(r^2/p_i)^{\frac{\log_2(p_i)}{4}}}=
2p_i^{\frac{\log_2(p_i)}{4}}r^{1-\frac{\log_2(p_i)}{2}}=
2p_i^{\frac{\log_2(p_i)}{4}}p_i^{1-\frac{\log_2(p_i)}{2}}\left(\frac{r}{p_i}\right)^{1-\frac{\log_2(p_i)}{2}}\\\nonumber
&\le  2 p_i^{1-\frac{\log_2(p_i)}{4}}2^{1-\frac{\log_2(p_i)}{2}}=
2^{2-\frac{\log_2(p_i)}{2}}p_i^{1-\frac{\log_2(p_i)}{4}}=2^{2+\frac{\log_2(p_i)}{2}-\frac{(\log_2(p_i))^2}{4}}
.
\end{align}
Set $x:=\log_2(p_i)$. Now, the expression $$2+\frac{x}{2}-\frac{x^2}{4}$$
is a parabola and it can be verified that it is negative when $x> 4$, that is, $p_i> 16$. In particular,~\eqref{eq:vera} follows from~\eqref{friday111} when $p_i> 16$.

Suppose $p_i\le 16$, that is, $p_i\le 13$ because $p_i$ is a prime number. Set $y=\log_2(r/p_i)$. By following the same computations as in~\eqref{friday111}, we obtain
\begin{align*}
\frac{S(p_i,a_i)r}{(r^2/p_i^{a_i})^{a_i\frac{\log_2(p_i)}{4}}}&
=2p_i^{1-\frac{\log_2(p_i)}{4}}2^{\left(1-\frac{\log_2(p_i)}{2}\right)y}=2^{1+\log_2(p_i)-\frac{(\log_2(p_i))^2}{4}+
\left(1-\frac{\log_2(p_i)}{2}\right)y
}
.
\end{align*}Now, the expression $$
1+\log_2(p_i)-\frac{(\log_2(p_i))^2}{4}+
\left(1-\frac{\log_2(p_i)}{2}\right)y
$$
is linear in $y$. When $p_i\ge 5$, we have $1-\log_2(p_i)/2<0$ and hence this linear expression is negative for $$y> \frac{\frac{(\log_2(p_i))^2}{4}-\log_2(p_i)-1}{1-\frac{\log_2(p_i)}{2}}.$$When $p_i=13$, we find $r/p_i=2^y> 2.8$ and hence~\eqref{eq:vera} is satisfied as long as $r/p_i> 2.8$. Thus, the only exceptions arising with $p_i=13$ are in~\eqref{appendix0}. All other cases are dealt with similarly, considering all primes $p_i\le 13$.

\smallskip

\noindent\textsc{Case $a_i=2$.} We have $S(p_i,a_i)=p_i+3$. 
When $p_i=2$, we have the exceptions in~\eqref{appendix1}. Therefore, for the rest of the argument we suppose $p_i\ge 3$. In particular, $p_i+3\le 2p_i$.
 As $\ell\ge2$, we have $r/p_i^{a_i}\ge 2$. Hence, we have
\begin{align}\label{friday}
\frac{(p_i+3)r}{(r^2/p_i^2)^{\frac{\log_2(p_i)}{2}}}&\le\frac{2p_i\cdot r}{(r^2/p_i^2)^{\frac{\log_2(p_i)}{2}}}=
2p_i^{1+\log_2(p_i)}r^{1-\log_2(p_i)}=
2p_i^{1+\log_2(p_i)}p_i^{2-2\log_2(p_i)}\left(\frac{r}{p_i^2}\right)^{1-\log_2(p_i)}\\\nonumber
&=2p_i^{3-\log_2(p_i)}\left(\frac{r}{p_i^2}\right)^{1-\log_2(p_i)}\le
2p_i^{3-\log_2(p_i)}2^{1-\log_2(p_i)}=
2^{2-\log_2(p_i)}p_i^{3-\log_2(p_i)} 
.
\end{align}
When $p_i\ge 7$, we have $2^{2-\log_2(p_i)}p_i^{3-\log_2(p_i)}<1$ and hence~\eqref{eq:vera} is satisfied.

 Assume $p_i=5$. If $r/p_i^2\ge 4$, then we may follow step by step the computations in~\eqref{friday} and replace $r/p_i^2$ with $4$ (rather than $2$); thus we obtain
$$\frac{(p_i+3)r}{(r^{2}/p_i^{a_i})^{a_i\frac{\log_2(p_i)}{4}}}\le 2^{3-2\log_2(p_i)}p_i^{3-\log_2(p_i)}<1.$$ If $r/p_i^2<4$, then $r=2\cdot p_i^2=50$ or $r=3\cdot p_i^2=75$ and we obtain the exceptions in~\eqref{appendix1}. 

Assume $p_i=3$. If $r/p_i^2\ge 64$, as above, we may follow step by step the computations in~\eqref{friday} and replace $r/p_i^2$ with $64=2^6$; thus we obtain
$$\frac{(p_i+3)r}{(r^{2}/p_i^{a_i})^{a_i\frac{\log_2(p_i)}{4}}}\le 2^{7-6\log_2(p_i)}p_i^{3-\log_2(p_i)}=0.83<1.$$ When $r/p_i^2<64$, we have computed explicitly the value on the left hand side of~\eqref{eq:vera} and we have verified that the only exceptions arise with $r/p_i^2\le 46$, that is, we obtain the exceptions in~\eqref{appendix1}.

\smallskip

\noindent\textsc{Case $a_i=3$.} We have $S(p_i,a_i)=2p_i^2+2p_i+4$. Moreover, we have
\begin{align*}
\frac{(2p_i^2+2p_i+4)r}{(r^2/p_i^3)^{3\frac{\log_2(p_i)}{4}}}
&\le
\frac{4p_i^2\cdot r}{(r^2/p_i^3)^{3\frac{\log_2(p_i)}{4}}}=
2^2p_i^{2+9\frac{\log_2(p_i)}{4}}r^{1-3\frac{\log_2(p_i)}{2}}=
2^2p_i^{2+9\frac{\log_2(p_i)}{4}}p_i^{3-9\frac{\log_2(p_i)}{2}}\left(\frac{r}{p_i^3}\right)^{1-3\frac{\log_2(p_i)}{2}}\\\nonumber
&=2^2p_i^{5-9\frac{\log_2(p_i)}{4}}\left(\frac{r}{p_i^3}\right)^{1-3\frac{\log_2(p_i)}{2}}\le
2^2p_i^{5-9\frac{\log_2(p_i)}{4}}2^{1-3\frac{\log_2(p_i)}{2}}=
2^{3-3\frac{\log_2(p_i)}{2}}p_i^{5-9\frac{\log_2(p_i)}{4}}\\
&=2^{3+7\frac{\log_2(p_i)}{2}-9\frac{(\log_2(p_i))^2}{4}} 
.
\end{align*}
This number is less then $1$ for each $p_i>3$ and hence~\eqref{eq:vera} is satisfied in these cases.

Assume $p_i=3$. If $r/p_i^2\ge 16$, as usual, we may follow the computations above but replacing $r/p_i^3$ with $16=2^4$; thus we obtain
$$\frac{(2p_i^2+2p_i+4)r}{(r^{2}/p_i^{a_i})^{a_i\frac{\log_2(p_i)}{4}}}\le 2^2p_i^{5-9\frac{\log_2(p_i)}{4}}2^{4-6\log_2(p_i)}=0.42<1.$$ When $r/p_i^3<16$, we have computed explicitly the value on the left hand side of~\eqref{eq:vera} and we have verified that the only exceptions arise with $r/p_i^3\le 7$, that is, we obtain the exceptions in~\eqref{appendix2}. 

Assume $p_i=2$. We obtain
$$\frac{(2p_i^2+2p_i+4)r}{(r^{2}/p_i^{a_i})^{a_i\frac{\log_2(p_i)}{4}}}=\frac{16r}{(r^2/8)^{3/4}}= \frac{2^{25/4}}{r^{1/2}}.$$ This expression is less than $1$ when $r\ge 5\, 800$, therefore the only exceptions arise when $r/p_i^3\le 723$, that is, we obtain the exceptions in~\eqref{appendix2}. 

\smallskip

\noindent\textsc{Case $a_i\in \{4,5\}$.}
These values of $a_i$ are checked similarly; indeed, for $a_i\in \{4,5\}$,~\eqref{eq:vera} is satisfied, except for the cases listed in~\eqref{appendix3} and~\eqref{appendix4}.

\smallskip

\noindent\textsc{Case $a_i\ge 6$.} We have $S(p_i,a_i)=c(p_i)p_i^{a_i^2/4}$. Assume $p_i\ge 3$. Using $c(p_i)\le c(3)<4$, we have
\begin{align*}
\frac{c(p_i)p_i^{\frac{a_i^2}{4}}r}{(r^2/p_i^{a_i})^{a_i\frac{\log_2(p_i)}{4}}}
&=
c(p_i)p_i^{\frac{a_i^2}{4}+\frac{a_i^2\log_2(p_i)}{4}}r^{1-\frac{a_i\log_2(p_i)}{2}}=
c(p_i)p_i^{\frac{a_i^2}{4}+\frac{a_i^2\log_2(p_i)}{4}+a_i-\frac{a_i^2\log_2(p_i)}{2}}\left(\frac{r}{p_i^{a_i}}\right)^{1-\frac{a_i\log_2(p_i)}{2}}\\
&<c(3)p_i^{\frac{a_i^2}{4}-\frac{a_i^2\log_2(p_i)}{4}+a_i}2^{1-\frac{a_i\log_2(p_i)}{2}}
=2^{3-\frac{a_i\log_2(p_i)}{2}}p_i^{\frac{a_i^2}{4}-\frac{a_i^2\log_2(p_i)}{4}+a_i}.
\end{align*}
Summing up,
\begin{equation}\label{saturday1}
\frac{S(p_i,a_i)r}{(r^{2}/p_i^{a_i})^{a_i\frac{\log_2(p_i)}{4}}}\le 2^{3-\frac{a_i\log_2(p_i)}{2}}p_i^{\frac{a_i^2}{4}-\frac{a_i^2\log_2(p_i)}{4}+a_i}.
\end{equation}

Now, $$3-\frac{a_i\log_2(p_i)}{2}\le 3-\frac{a_i\log_2(3)}{2}\le 3-\frac{6\cdot\log_2(3)}{2}<0.$$ 
Similarly, for $(a_i,p_i)\ne (6,3)$, one can check that $$\frac{a_i^2}{4}-\frac{a_i^2\log_2(p_i)}{4}+a_i<0.$$
Therefore, for $(a_i,p_i)\ne (6,3)$,~\eqref{eq:vera} follows from~\eqref{saturday1}. Finally, when $(a_i,p_i)=(6,3)$, we have 
$$2^{3-\frac{a_i\log_2(p_i)}{2}}p_i^{\frac{a_i^2}{4}-\frac{a_i^2\log_2(p_i)}{4}+a_i}=0.6646<1$$
and hence~\eqref{eq:vera} follows again from~\eqref{saturday1}.

Finally, assume $p_i=2$ and set $r'=r/p_i^{a_i}$. We have
\begin{align*}
\frac{c(p_i)p_i^{\frac{a_i^2}{4}}r}{(r^2/p_i^{a_i})^{a_i\frac{\log_2(p_i)}{4}}}
&=
c(2)2^{\frac{a_i^2}{4}+\frac{a_i^2}{4}}r^{1-\frac{a_i}{2}}=
c(2)2^{\frac{a_i^2}{2}}r^{1-\frac{a_i}{2}}=c(2)2^{\frac{a_i^2}{2}}2^{a_i-\frac{a_i^2}{2}}r'^{1-\frac{a_i}{2}}\\
&=c(2)2^{a_i}r'^{1-\frac{a_i}{2}}=c(2)2^{a_i}2^{(1-a_i/2)\log_2 r'}<2^{3+a_i+(1-a_i/2)\log_2(r')}.
\end{align*}
Now, consider the function $g(a_i,r')=3+a_i+(1-a_i/2)\log_2 (r')$, where we think of $a_i$ and $r'$ as continuous variables. We have $\partial g/\partial{a_i}=1-\log_2(r')/2\le 0$. Therefore,
$$g(a_i,r')\le 3+6+(1-6/2)\log_2(r')=9-2\log_2(r').$$
Now, $9-2\log_2(r')\le 0$, when $\log_2(r')\ge 9/2$, that is, $r'\ge 22$. In particular, the only exceptions to~\eqref{eq:vera} arise when $r/p_i^{a_i}=r'\le 21$, as stated in~\eqref{appendix5}.

\subsection{Directional monotonicity of $f(r)$ and proof of Lemma~\ref{direction1}}\label{sec:directionalmonotonicity}Let $i\in \{1,\ldots,\ell\}$. By considering the discrete variable $p_i$ as continue, we find
\begin{equation*}
\frac{\partial f(r)}{\partial p_i}=-\frac{a_i\log c(2)}{p_i(\log r)^2}+\frac{a_i}{4p_i\log 2 }-\frac{\partial S(p_i,a_i)/\partial p_i}{S(p_i,a_i)\log r}+\frac{a_i\log S(p_i,a_i)}{p_i(\log r)^2}-\frac{(\ell-2)a_i\log 2}{p_i(\log r)^2}.
\end{equation*}
We are interested in showing $\partial f(r)/\partial p_i\ge 0$, $\forall p_i\ge 2$, where $f(r)=f(p_1,\ldots,p_\ell)$ is thought as a function in $p_1,\ldots,p_\ell$ with $a_1,\ldots,a_\ell$ being fixed. From this, the proof of Lemma~\ref{direction1} immediately follows.

Multiplying by $p_i(\log r)^2$, we obtain
\begin{equation}\label{partial2}
p_i(\log r)^2\frac{\partial f(r)}{\partial p_i}=-a_i\log c(2)+\frac{a_i(\log r)^2}{4\log 2 }-\frac{p_i\partial S(p_i,a_i)/\partial p_i}{S(p_i,a_i)}\log r+a_i\log S(p_i,a_i)-(\ell-2)a_i\log 2.
\end{equation}
We now distinguish various cases depending on $a_i$. 

\smallskip

\noindent\textsc{Case $a_i=1$.} Then $S(p_i,a_i)=2$ and, from~\eqref{partial2}, we obtain
\begin{equation}\label{eq:veraveraveraserbia}
p_i(\log r)^2\frac{\partial f(r)}{\partial p_i}=\log (8/c(2))+\frac{(\log r)^2}{4\log 2 }-\ell\log 2.
\end{equation}
We have
\begin{equation}\label{vera:not here}
\ell\le \log_2 r=\frac{\log r}{\log 2},
\end{equation}
because $\ell$ is the number of prime factors of $r$. From~\eqref{eq:veraveraveraserbia} and~\eqref{vera:not here} and from $c(2)<8$, we have
\begin{equation*}
p_i(\log r)^2\frac{\partial f(r)}{\partial p_i}\ge \log (8/c(2))+\frac{(\log r)^2}{4\log 2 }-\log r\ge \frac{(\log r)^2}{4\log 2}-\log r.
\end{equation*}
The formula appearing on the right hand side of $p_i(\log r)^2\partial f(r)/\partial p_i$ is positive for $\log r\ge 4\log 2$, that is, $r\ge 16$. In particular, $\partial f(r)/\partial p_i\ge 0$, for each $r$ with $r\ge 16$. Recalling that $\ell\ge 3$, the condition $r\ge 16$ is automatically satisfied.

\smallskip

\noindent\textsc{Case $a_i=2$.} Then $S(p_i,a_i)=p_i+3$ and $\partial S(p_i,a_i)/\partial p_i=1$ and, from~\eqref{partial2}, we obtain
\begin{equation*}
p_i(\log r)^2\frac{\partial f(r)}{\partial p_i}=-2\log c(2)+\frac{(\log r)^2}{2\log 2 }-\log r\frac{p_i}{p_i+3}+2\log(p_i+3)-2(\ell-2)\log 2.
\end{equation*}
As above, by using $p_i\ge 2$ in the first inequality, by using $\log(16/c(2)^2)+2\log(5)\ge 1.99$ in the second inequality and by using~\eqref{vera:not here} in the last inequality, we deduce
\begin{align*}
p_i(\log r)^2\frac{\partial f(r)}{\partial p_i}&= \log (16/c(2)^2)+\frac{(\log r)^2}{2\log 2 }-\log r\frac{p_i}{p_i+3}+2\log(p_i+3)-2\ell\log 2\\
&\ge \log (16/c(2)^2)+\frac{(\log r)^2}{2\log 2 }-\log r+2\log(5)-2\ell\log 2\\
&\ge \frac{(\log r)^2}{2\log 2 }-\log r-2\ell\log 2+1.99\ge
\frac{(\log r)^2}{2\log 2 }-\log r-2\log r+1.99=\frac{(\log r)^2}{2\log 2 }-3\log r+1.99.
\end{align*}The formula appearing on the right hand side of $p_i(\log r)^2\partial f(r)/\partial p_i$ is quadratic in $\log r$ and it can be verified that it is positive for $r\ge 28$.   As $\ell\ge 3$ by~\eqref{C}, we have $r\ge 2\cdot 3\cdot 5=30$ and hence the condition $r\ge 28$ is automatically satisfied. 
Thus  $\partial f(r)/\partial p_i\ge 0$, $\forall p_i\ge 2$. 

\smallskip

\noindent\textsc{Case $a_i=3$.} Then $S(p_i,a_i)=2p_i^2+2p_i+4$ and $\partial S(p_i,a_2)/\partial p_i=4p_i+2$ and, from~\eqref{partial2}, we obtain
\begin{equation*}
p_i(\log r)^2\frac{\partial f(r)}{\partial p_i}=-3\log c(2)+\frac{3(\log r)^2}{4\log 2 }-\frac{4p_i^2+2p_i}{2p_i^2+2p_i+4}\log r+3\log(2p_i^2+2p_i+4)-3(\ell-2)\log 2.
\end{equation*}
We deduce
\begin{align*}
p_i(\log r)^2\frac{\partial f(r)}{\partial p_i}&\ge-3\log c(2)+\frac{3(\log r)^2}{4\log 2 }-2\log r+3\log(2p_i^2+2p_i+4)-3(\ell-2)\log 2\\
&\ge -3\log c(2)+\frac{3(\log r)^2}{4\log 2 }-2\log r+3\log(16)-3(\ell-2)\log 2\\
&=-3\log c(2)+\frac{3(\log r)^2}{4\log 2 }-2\log r+18\log(2)-3\ell\log 2\\
&\ge -3\log c(2)+\frac{3(\log r)^2}{4\log 2 }-2\log r+18\log(2)-3\log r\\
&=-3\log c(2)+\frac{3(\log r)^2}{4\log 2 }-5\log r+18\log(2).
\end{align*}
Now, using this expression, it can be verified that  $\partial f(r)/\partial p_i$ is always positive.

   
\smallskip

\noindent\textsc{Case $a_i\in \{4,5\}$.} These cases are analogous and their proof is omitted.

\smallskip

\noindent\textsc{Case $a_i\ge 6$.} Then $S(p_i,a_i)=c(p_i)p_i^{\frac{a_i^2}{4}}$ and 
$$
\frac{\partial S(p_i,a_i)}{\partial p_i}=c(p_i)\frac{a_i^2}{4}p_i^{\frac{a_i^2}{4}-1}+\frac{\partial c(p_i)}{\partial p_i}p_i^{\frac{a_i^2}{4}}\le 
c(p_i)\frac{a_i^2}{4}p_i^{\frac{a_i^2}{4}-1},$$
because $c(p_i)$ is a strictly decreasing function. Thus, from~\eqref{partial2}, we deduce
\begin{equation*}
p_i(\log r)^2\frac{\partial f(r)}{\partial p_i}\ge -a_i\log c(2)+\frac{a_i(\log r)^2}{4\log 2 }-\frac{a_i^2}{4}\log r+a_i\log S(p_i,a_i)-(\ell-2)a_i\log 2.
\end{equation*}
Using~\eqref{vera:not here}, we deduce
\begin{align}\label{eq:refine}
p_i(\log r)^2\frac{\partial f(r)}{\partial p_i}&\ge-a_i\log (c(2)/4)+\frac{a_i(\log r)^2}{4\log 2 }-\frac{a_i^2}{4}\log r+a_i\log S(p_i,a_i)-a_i\log r\\\nonumber
&\ge \frac{a_i(\log r)^2}{4\log 2 }-\frac{a_i^2}{4}\log r-a_i\log r,
\end{align}
where in the last inequality we are using $S(p_i,a_i)\ge S(2,6)\ge 2^9$ and hence $\log S(p_i,a_i)\ge \log(c(2)/4)$. Therefore,
\begin{equation}\label{hup}\frac{p_i\log r}{a_i}\frac{\partial f(r)}{\partial p_i}\ge\frac{\log r}{4\log 2}-\frac{a_i}{4}-1.\end{equation}
When $p_i\ge 3$, we have $r\ge 3^{a_i}\cdot 2\cdot 5$, because $\ell\ge 3$. Hence $\log r\ge a_i\log 3+\log 10$. Therefore, $$\frac{\log r}{4\log 2}\ge \frac{a_i\log 3}{4\log 2}+\frac{\log 10}{4\log 2}.$$
This inequality and the fact that $a_i\ge 6$ show that~\eqref{hup} is satisfied and hence $\partial f(r)/\partial p_i\ge 0$.
Assume then $p_i=2$. When $r/2^{a_i}>15$, we have  $\log r\ge a_i\log 3+\log 16$. Therefore, $$\frac{\log r}{4\log 2}\ge \frac{a_i}{4}+\frac{\log 16}{4\log 2}=\frac{a_i}{4}+1$$
and hence~\eqref{hup} shows that $\partial f(r)/\partial p_i\ge 0$.
Finally, when $p_i=2$ and $r/2^{a_i}<16$, we must have $r=2^{a_i}\cdot 3\cdot 5$, because $\ell\ge 3$. In this particular case, we may refine the computations in~\eqref{eq:refine}. Indeed, by repeating the same steps except by replacing $S(p_i,a_i)$ with $c(2)2^{a_i^2/4}$, we have
\begin{align}\label{eq:refine1}
p_i(\log r)^2\frac{\partial f(r)}{\partial p_i}&\ge-a_i\log c(2)+\frac{a_i(\log r)^2}{4\log 2 }-\frac{a_i^2}{4}\log r+a_i\log S(p_i,a_i)-(\ell-2)a_i\log 2\\\nonumber
&\ge-a_i\log c(2)+\frac{a_i(\log r)^2}{4\log 2 }-\frac{a_i^2}{4}\log r+a_i\log c(2)+\frac{a_i^3}{4}\log 2-a_i\log r+4\log 2\\\nonumber
&\ge \frac{a_i(\log r)^2}{4\log 2 }+\frac{a_i^3}{4}\log 2-\frac{a_i^2}{4}\log r-a_i\log r.
\end{align}
Therefore,
$$\frac{p_i\log r}{a_i}\frac{\partial f(r)}{\partial p_i}\ge  \frac{\log r}{4\log 2 }+\frac{a_i^2}{4}\frac{\log 2}{\log r}-\frac{a_i}{4}-1.$$
It is now clear that  also in this case $\partial f(r)/\partial p_i\ge 0$.

\begin{proof}[Proof of Lemma~$\ref{direction1}$]
We have proved that the directional derivative $\partial/\partial p_i$ of $f(r)$ is positive. Since $r$ is good, we have $f(r)> 0$. Therefore, as $p'>p_i$, $f(r')\ge f(r)> 0$ and $r'$ is also good.
\end{proof}

\subsection{Increasing the number of prime factors and proof of Lemma~\ref{direction2}}\label{sec:direction2}
Let $p$ be a prime number with $p\notin \{p_1,\ldots,p_\ell\}$ and let $r'=r\cdot p$. Observe that $r$ has $\ell$ distinct prime factors, whereas $r'$ has $\ell+1$ distinct prime factors. We have
\begin{align*}
f(r')-f(r)=&\log c(2)\left(\frac{1}{\log r'}-\frac{1}{\log r}\right)+\frac{\log p}{4\log 2}-\frac{\log 2}{\log r'}-\sum_{i=1}^\ell \log S(p_i,a_i)\left(\frac{1}{\log r'}-\frac{1}{\log r}\right)\\
&-1+(\ell-1)\frac{\log 2}{\log r'}-(\ell-2)\frac{\log 2}{\log r}\\
=&-\frac{\log c(2)\log p}{\log r \log r'}+\frac{\log p}{4\log 2}-\sum_{i=1}^\ell \log S(p_i,a_i)\left(\frac{1}{\log r'}-\frac{1}{\log r}\right)-1+(\ell-2)\frac{\log 2}{\log r'}-(\ell-2)\frac{\log 2}{\log r}\\
=&-\frac{\log c(2)\log p}{\log r \log r'}+\frac{\log p}{4\log 2}+\left(\sum_{i=1}^\ell \log S(p_i,a_i)\right)\frac{\log p}{\log r\log r'}-1-(\ell-2)\frac{\log 2\log p}{\log r\log r'}.
\end{align*}
Now, when $i\ge 2$, we have $S(p_i,a_i)\ge 2$ and hence $\log S(p_i,a_i)\ge  \log 2$. 
Similarly, when $i=1$, we have $p_1=2$ and $a_1\ge 2$ by~\eqref{B}; hence $S(p_i,a_i)\ge S(2,2)=5\ge 4$; thus $\log S(p_i,a_i)\ge  2\log 2$. Taking this into account, we obtain
\begin{align*}
f(r')-f(r)&\ge -\frac{\log c(2)\log p}{\log r \log r'}+\frac{\log p}{4\log 2}+\frac{(\ell+1)\log 2\log p}{\log r\log r'}-1-(\ell-2)\frac{\log 2\log p}{\log r\log r'}\\
&=
-\frac{\log c(2)\log p}{\log r \log r'}+\frac{\log p}{4\log 2}+\frac{3\log 2\log p}{\log r\log r'}-1
\ge \frac{\log p}{4\log 2}-1,
\end{align*}
where in the last inequality we have used $8>c(2)$, that is, $3\log 2\ge \log c(2)$. 

\begin{proof}[Proof of Lemma~$\ref{direction2}$]
As  $p\ge 17$, we have $\log p/(4\log 2)-1\ge 0$. Therefore, $f(r')-f(r)\ge \log p/(4\log 2)-1\ge 0$. Since $r$ is good, we have $f(r)> 0$. Thus $f(r')> 0$ and $r'$ is good.
\end{proof}

\subsection{Increasing the multiplicity of a prime}\label{sec:direction3}
Let $i\in \{1,\ldots,\ell\}$. We think of $a_i$ as a continuous variable and we compute $\partial \mathtt{f}(r)/\partial a_i$. We are interested in showing $\partial\mathtt{f}(r)/\partial a_i\ge 0$, where $f(r)=f(a_1,\ldots,a_\ell)$ is thought as a function of $a_1,\ldots,a_\ell$ with $p_1,\ldots,p_\ell$ being fixed. From this Lemma~\ref{direction3} immediately follows.

We have
\begin{align*}
\frac{\partial \mathtt{f}(r)}{\partial a_i}&=
-\frac{\log p_i\log c(2)}{(\log r)^2}+\frac{\log p_i}{4\log 2}-\frac{a_i\log p_i}{2\log r}+\frac{\log p_i\log (c(p_i)p_i^{a_i^2/4})}{(\log r)^2}-(\ell-2)\frac{\log 2\log p_i}{(\log r)^2}.
\end{align*}
Multiplying this by $(\log r)^2$, we obtain
\begin{align}\label{wednesdayevening}
(\log r)^2\frac{\partial \mathtt{f}(r)}{\partial a_i}&=
-\log p_i\log c(2)
+
\frac{\log p_i}{4\log 2}(\log r)^2-\frac{a_i\log p_i}{2}\log r
+
\log p_i\log(c(p_i)p_i^{a_i^2/4})
-
(\ell-2)\log 2\log p_i.
\end{align}
In particular, using~\eqref{vera:not here} and $c(p_i)\ge 1$, we get
\begin{align}\label{vera:not cool}
\frac{(\log r)^2}{\log p_i}\frac{\partial \mathtt{f}(r)}{\partial a_i}&\ge
-\log (c(2)/4)
+
\frac{(\log r)^2}{4\log 2}-\frac{a_i}{2}\log r
+
\frac{a_i^2\log p_i}{4}
-
\log r.
\end{align}
Write $r=p_i^{a_i}r'$. Replacing $r$ in~\eqref{vera:not cool} with $p_i^{a_i}r'$, we obtain
\begin{align*}
\frac{(\log r)^2}{\log p_i}\frac{\partial \mathtt{f}(r)}{\partial a_i}\ge&
-\log (c(2)/4)
+
\frac{a_i^2(\log p_i)^2}{4\log 2}+\frac{(\log r')^2}{4\log 2}+\frac{a_i\log p_i\log r'}{2\log 2}
\\&-\frac{a_i^2\log p_i}{2}-\frac{a_i\log r'}{2}
+
\frac{a_i^2\log p_i}{4}
-
a_i\log p_i-\log r'\\
=&
\left(
-\log (c(2)/4)+\frac{(\log r')^2}{4\log 2}-\log r'
\right)
+\left(
\frac{a_i^2(\log p_i)^2}{4\log 2}-\frac{a_i^2\log p_i}{4}
\right)\\
&+\frac{a_i\log p_i\log r'}{2\log 2}-\frac{a_i\log r'}{2}-a_i\log p_i
.
\end{align*}
Now, $\log p_i\ge \log 2$ and hence
$$\frac{a_i^2(\log p_i)^2}{4\log 2}-\frac{a_i^2\log p_i}{2}+\frac{a_i^2\log p_i}{4}\ge 0.$$

Assume, for the time being, $p_i\ge 3$ and $r'\ge 27$.
If we set $x:=\log r'$, then 
$$-\log (c(2)/4)+\frac{(\log r')^2}{4\log 2}-\log r'=-\log(c(2)/4)+\frac{x^2}{4\log 2}-x$$
is a parabola and it is not hard to verify that it is positive when $r'\ge 27$.
Furthermore, 
$$\frac{a_i\log p_i\log r'}{2\log 2}-\frac{a_i\log r'}{2}-a_i\log p_i\ge a_i\left(\frac{\log 3\log r'}{2\log 2}-\frac{r'}{2}-\log 3\right)\ge 0,$$
where in the last inequality we are using again $r'\ge 27$. In particular, we have shown that $\partial \mathtt{f}(r)/\partial a_i\ge 0$, provided that $r'\ge 27$ and $p_i\ge 3$.

Assume now $p_i=2$. Thus $i=1$ and $r=2^{a_1}r'$. By specializing~\eqref{wednesdayevening} with $p_i=2$ and by using $\log r=a_1\log 2+\log r'$, we obtain
\begin{align*}
\frac{(\log r)^2}{\log 2}\frac{\partial \mathtt{f}(r)}{\partial a_i}&=
-\log c(2)
+
\frac{(\log r)^2}{4\log 2}-\frac{a_i\log r}{2}
+\log c(2)+\frac{a_i^2\log 2}{4}
-
(\ell-2)\log 2\\
&=
\frac{(\log r)^2}{4\log 2}-\frac{a_i\log r}{2}
+\frac{a_i^2\log 2}{4}
-
(\ell-2)\log 2\\
&=
\frac{a_1^2\log 2}{4}+\frac{(\log r')^2}{4\log 2}+\frac{a_1\log r'}{2}-\frac{a_1^2\log 2}{2}-\frac{a_1\log r'}{2}+\frac{a_1^2\log 2}{4}-(\ell-2)\log 2\\
&=
\frac{(\log r')^2}{4\log 2}-(\ell-2)\log 2.
\end{align*}
Now, $r'=p_2^{a_2}\cdots p_\ell^{a_\ell}$ with $p_2,\ldots,p_\ell\ge 3$ and hence $\log r'\ge \ell-1$. We deduce
$$\frac{(\log r)^2}{\log 2}\frac{\partial \mathtt{f}(r)}{\partial a_i}\ge \frac{(\log r')^2}{4\log 2}-(\log r'-1)\log 2.$$
The expression appearing on the right hand side is a parabola in $x:=\log r'$ and it is not hard to verify that it is positive. Therefore, $\partial\mathtt{f}(r)/\partial a_i\ge 0$ also in this case.

Finally, suppose $p_i\ge 3$ and $r'< 27$. Since the product of three distinct primes is at least $30$, we deduce $\ell=3$. Assume, $r'\ge 10$. By specializing~\eqref{wednesdayevening} with $\ell=3$ and by using $$a_i\log p_i+\log 10\le \log r=a_i\log p_i+\log r'\le a_i\log p_i+\log 26,$$ 
we obtain
\begin{align*}
\frac{(\log r)^2}{\log p_i}\frac{\partial \mathtt{f}(r)}{\partial a_i}&=
-\log c(2)
+
\frac{(\log r)^2}{4\log 2}
-\frac{a_i\log r}{2}
+\frac{a_i^2\log p_i}{4}
-
\log 2\\
&=
-\log (c(2)/2)
+
\frac{(\log r)^2}{4\log 2}
-\frac{a_i\log r}{2}
+\frac{a_i^2\log p_i}{4}\\
&\ge 
-\log (c(2)/2)
+
\frac{a_i^2(\log p_i)^2}{4\log 2}+\frac{(\log 10)^2}{4\log 2}+\frac{a_i\log p_i\log 10}{2\log 2}
-\frac{a_i^2\log p_i}{2}-\frac{a_i\log 26}{2}
+\frac{a_i^2\log p_i}{4}\\
&=
-\log (c(2)/2)
+
\frac{a_i^2(\log p_i)^2}{4\log 2}+\frac{(\log 10)^2}{4\log 2}+\frac{a_i\log p_i\log 10}{2\log 2}
-\frac{a_i^2\log p_i}{4}-\frac{a_i\log 26}{2}\\
&\ge 
\frac{a_i^2(\log p_i)^2}{4\log 2}+\frac{a_i\log p_i\log 10}{2\log 2}
-\frac{a_i^2\log p_i}{4}-\frac{a_i\log 26}{2},
\end{align*}
where the last inequality follows because $-\log(c(2)/2)+(\log 10)^2/(4\log 2)>0$. Observe that
\begin{align*}
\frac{a_i^2(\log p_i)^2}{4\log 2}-\frac{a_i^2\log p_i}{4}&=\frac{a_i^2\log p_i}{4}\left(\frac{\log p_i}{\log 2}-1\right)\ge 0,
\end{align*}
because $p_i\ge 3$. Observe also that
\begin{align*}
\frac{a_i\log p_i\log 10}{2\log 2}
-\frac{a_i\log 26}{2}
&=
\frac{a_i}{2}\left(\frac{\log p_i\log 10}{\log 2}-\log 26\right)\ge 0,
\end{align*}
because $p_i\ge 3$. Therefore $\partial\mathtt{f}(r)/\partial a_i\ge 0$ in this case. It remains to consider the case $r'<10$. Since $\ell=3$, $r'$ is the product of two distinct primes and hence $r'=6$. In other words, $i=\ell=3$ and $r=2\cdot 3\cdot p_i^{a_i}$. By repeating the computations above with this value of $r$ and $r'$, we obtain
\begin{align*}
\frac{(\log r)^2}{\log p_i}\frac{\partial \mathtt{f}(r)}{\partial a_i}&=
-\log (c(2)/2)
+
\frac{(\log r)^2}{4\log 2}
-\frac{a_i\log r}{2}
+\frac{a_i^2\log p_i}{4}\\
&\ge 
-\log (c(2)/2)
+
\frac{a_i^2(\log p_i)^2}{4\log 2}+\frac{(\log 6)^2}{4\log 2}+\frac{a_i\log p_i\log 6}{2\log 2}
-\frac{a_i^2\log p_i}{2}-\frac{a_i\log 6}{2}
+\frac{a_i^2\log p_i}{4}\\
&=
-\log (c(2)/2)
+
\frac{a_i^2(\log p_i)^2}{4\log 2}+\frac{(\log 6)^2}{4\log 2}+\frac{a_i\log p_i\log 6}{2\log 2}
-\frac{a_i^2\log p_i}{4}-\frac{a_i\log 6}{2}.
\end{align*}
Now, using $a_i\ge 1$ and $p_i\ge 5$, we have
$$\frac{a_i\log p_i\log 6}{2\log 2}-\frac{a_i\log 6}{2}=\frac{a_i\log 6}{2}\left(
\frac{\log p_i}{\log 2}-1
\right)\ge \frac{a_i\log 6}{2}\ge \frac{\log 6}{2}.$$
Therefore,
\begin{align*}
\frac{(\log r)^2}{\log p_i}\frac{\partial \mathtt{f}(r)}{\partial a_i}
&\ge
-\log (c(2)/2)+\frac{\log 6}{2}+\frac{(\log 6)^2}{4\log 2}
+
\frac{a_i^2(\log p_i)^2}{4\log 2}
-\frac{a_i^2\log p_i}{4}\\
&\ge
\frac{a_i^2(\log p_i)^2}{4\log 2}
-\frac{a_i^2\log p_i}{4}=\frac{a_i^2\log p_i}{4}\left(\frac{\log p_i}{\log 2}-1\right)\ge 0,
\end{align*}
where the second inequality follows with a calculation and the third inequality follows because $p_i\ge 5$.

\begin{proof}[Proof of Lemma~$\ref{direction3}$]
From above $\partial \mathtt{f}/\partial a_i\ge 0$. Thus if $\mathtt{f}(r)> 0$, then $\mathtt{f}(r')=\mathtt{f}(r\cdot p_i)\ge \mathtt{f}(r)> 0$ and hence $r'$ is $\mathtt{good}$.
\end{proof}
\subsection{Proof of Lemma~\ref{bounds:technical}}\label{bounds:technicalsec}
We start by proving part~\eqref{bounds:technical1}. As $r\ge 2\,000$, we have
\begin{align*}
(r/14)^{\frac{\log_2(r/14)}{4}}&=
2^{-\frac{\log_2(r/14)}{4}}
(r/7)^{\frac{\log_2(r/14)}{4}}\le 2^{-\frac{\log_2(2\,000/14)}{4}}(r/7)^{\frac{\log_2(r/14)}{4}} \\
&=2^{-\frac{\log_2(2\,000/14)}{4}}(r/7)^{-\frac{1}{4}}(r/7)^{\frac{\log_2(r/7)}{4}}\le
2^{-\frac{\log_2(2\,000/14)}{4}}(2\,000/7)^{-\frac{1}{4}}(r/7)^{\frac{\log_2(r/7)}{4}}\\
&\le 2^{-3}(r/7)^{\frac{\log_2(r/7)}{4}},
\end{align*}
where the last inequality follows with a computation. Therefore,
\begin{align*}
4r(r/7)^{\frac{\log_2(r/7)}{4}}+8r(r/14)^{\frac{\log_2(r/14)}{4}}\le 
5r(r/7)^{\frac{\log_2(r/7)}{4}}<5\cdot 2^{\log_2(r)}\cdot 2^{\frac{(\log_2(r/7))^2}{4}}=2^{\frac{(\log_2(r/7))^2}{4}+\log_2(r)+\log_2(5)}.
\end{align*}
Similarly, $r^{\log_2(r)/4}=2^{(\log_2(r))^2/4}$. We show that
\begin{equation}\label{mondaymonday}
\frac{(\log_2(r/7))^2}{4}+\log_2(r)+\log_2(5)\le\frac{(\log_2(r))^2}{4},
\end{equation}
from which~\eqref{bounds:technical1} immediately follows. Rearranging the summands in~\eqref{mondaymonday} and using $\log_2(r/7)=\log_2(r)-\log_2(7)$, we obtain the equivalent inequality
$$\left(\frac{\log_2(7)}{2}-1\right)\log_2(r)\ge \log_2(5)+\frac{(\log_2(7))^2}{4},$$
that is,
\begin{equation}\label{mondaynew}\log_2(r)\ge \frac{\log_2(5)+\frac{(\log_2(7))^2}{4}}{\frac{\log_2(7)}{2}-1}=10.63.
\end{equation}
As $r\ge 2\,000$, we have $\log_2(r)\ge 10.96$ and hence~\eqref{mondaynew} is satisfied.

We  prove part~\eqref{bounds:technical2}. We have
\begin{align*}
8r(r/28)^{\frac{\log_2(r/28)}{4}}+6r(r/21)^{\frac{\log_2(r/21)}{4}}+4r(r/14)^{\frac{\log_2(r/14)}{4}}&\le
8r(r/14)^{\frac{\log_2(r/14)}{4}}+6r(r/14)^{\frac{\log_2(r/14)}{4}}+4r(r/14)^{\frac{\log_2(r/14)}{4}}\\
&=18r(r/14)^{\frac{\log_2(r/14)}{4}}=2^{\frac{(\log_2(r/14))^2}{4}+\log_2(r)+\log_2(18)}.
\end{align*}
As above, we use the fact that $r^{\log_2(r)/4}=2^{(\log_2(r))^2/4}$. We show that
\begin{equation}\label{mondaymondayyy}
\frac{(\log_2(r/14))^2}{4}+\log_2(r)+\log_2(18)\le\frac{(\log_2(r))^2}{4},
\end{equation}
from which~\eqref{bounds:technical2} immediately follows. Rearranging the summands in~\eqref{mondaymondayyy} and using $\log_2(r/14)=\log_2(r)-\log_2(14)$, we obtain the equivalent inequality
$$\left(\frac{\log_2(14)}{2}-1\right)\log_2(r)\ge \log_2(18)+\frac{(\log_2(14))^2}{4},$$
that is,
\begin{equation}\label{mondayneww}\log_2(r)\ge \frac{\log_2(18)+\frac{(\log_2(14))^2}{4}}{\frac{\log_2(14)}{2}-1}=8.62.
\end{equation}
As  $\log_2(r)\ge 10.96$,~\eqref{mondayneww} is satisfied.

We  prove part~\eqref{bounds:technical3}. As $r\ge 2\,000$, we have
\begin{align*}
(r/10)^{\frac{\log_2(r/10)}{4}}&=
2^{-\frac{\log_2(r/10)}{4}}
(r/5)^{\frac{\log_2(r/10)}{4}}\le 2^{-\frac{\log_2(2\,000/10)}{4}}(r/5)^{\frac{\log_2(r/10)}{4}} \\
&=2^{-\frac{\log_2(200)}{4}}(r/5)^{-\frac{1}{4}}(r/5)^{\frac{\log_2(r/5)}{4}}\le
2^{-\frac{\log_2(200)}{4}}(2\,000/5)^{-\frac{1}{4}}(r/5)^{\frac{\log_2(r/5)}{4}}\\
&\le 2^{-4}(r/5)^{\frac{\log_2(r/5)}{4}},
\end{align*}
where the last inequality follows with a computation. Therefore,
\begin{align*}
(2r+10)(r/10)^{\frac{\log_2(r/10)}{4}}+(r/5+5)(r/5)^{\frac{\log_2(r/5)}{4}}
&\le
\left(\frac{13}{40}r+\frac{45}{8}\right)(r/5)^{\frac{\log_2(r/5)}{4}}\le 2^{-1}r(r/5)^{\frac{\log_2(r/5)}{4}}.\\
\end{align*}
Recall $r^{\log_2(r)/4}=2^{(\log_2(r))^2/4}$. We show that
\begin{equation}\label{mondaymondayyyo}
-1+\log_2(r)+\frac{(\log_2(r/5))^2}{4}\le\frac{(\log_2(r))^2}{4},
\end{equation}
from which~\eqref{bounds:technical3} immediately follows. Rearranging the summands in~\eqref{mondaymondayyyo} and using $\log_2(r/5)=\log_2(r)-\log_2(5)$, we obtain the equivalent inequality
$$\left(\frac{\log_2(5)}{2}-1\right)\log_2(r)\ge \frac{(\log_2(5))^2}{4}-1,$$
that is,
\begin{equation}\label{mondaynewww}\log_2(r)\ge \frac{-1+\frac{(\log_2(5))^2}{4}}{-1+\frac{\log_2(5)}{2}}=\frac{\log_{2}(5)}{2}+1=2.16.
\end{equation}
As $r\ge 2\,000$, we have $\log_2(r)\ge 10.96$ and hence~\eqref{mondaynewww} is satisfied.

We  prove part~\eqref{bounds:technical4}. We have
\begin{align*}
2p(r/p)^{\log_2(r/p)}&=
2^{1+\log_2(p)+\frac{(\log_2(r/p))^2}{4}}=2^{1+\log_2(r)+\frac{(\log_2(r))^2}{4}+\frac{(\log_2(p))^2}{4}-\frac{\log_2(p)\log_2(r)}{2}}.
\end{align*} 
Therefore,~\eqref{bounds:technical4} is equivalent to the inequality
\begin{equation*}
1+\log_2(r)+\frac{(\log_2(r))^2}{4}+\frac{(\log_2(p))^2}{4}-\frac{\log_2(p)\log_2(r)}{2}
\le\frac{(\log_2(r))^2}{4}.
\end{equation*}
In turn, this is equivalent to
$$\left(\frac{\log_2(p)}{2}-1\right)\log_2(r)\ge \frac{(\log_2(p))^2}{4}+1,$$
that is,
\begin{equation}\label{mondaynewwwx}\log_2(r)\ge \frac{1+\frac{(\log_2(p))^2}{4}}{-1+\frac{\log_2(p)}{2}}\ge
\frac{\log_2(p)}{2}+1.
\end{equation}
As $r\ge 2p$, we have $\log_2(r)\ge \log_2(p)+1$ and hence~\eqref{mondaynewwwx} is satisfied.

\end{document}